\documentclass[11pt]{article}
\usepackage{amsmath,amsthm}
\usepackage{pstricks,pst-node,pst-tree}
\usepackage{latexsym}
\usepackage{amssymb,amscd}
\usepackage{graphics}   
\usepackage{url}
\numberwithin{equation}{section}
\let \:=\colon
\let \beg=\begin

\let \hra=\hookrightarrow
\let \mb=\mathbb
\let \mc= \mathcal
\let \rt=\rightarrow
\let \ra=\rightarrow
\let \st=\stackrel
\let \La=\Lambda
\let \Ga=\Gamma

\let \de=\delta

\let \wt=\widetilde
\let \ga=\gamma
\let \Om=\Omega
\let \fl=\flushleft
\let \fr=\frac

\let \part=\partial
\let \pr=\prime

\let \sub=\subset

\newtheorem{lem}{Lemma}[section]

\newtheorem{thm}{Theorem}[section]

\begin{document}
\title{\large Rational curves of degree $11$ on a general quintic threefold \thanks{$2000$ {\it Mathematics Subject Classification}.
Primary 14J30; Secondary 13P10, 14H45.}}
\author{\small Ethan Cotterill \thanks{Research was conducted in part while the author was partially supported by an NSF graduate fellowship.}}
\vspace{-.5cm}
\date{\empty}
\thispagestyle{empty}
\maketitle
\vspace{-30pt}
\begin{abstract}
We prove that the incidence scheme of rational curves of degree 11 in quintic threefolds is irreducible. Irreducibility implies a strong form of the Clemens conjecture in degree
$11$. Namely, on a general quintic $F$ in $\mathbb{P}^4$,
there are only finitely many smooth rational curves of degree $11$, and
each curve $C$ is embedded in $F$ with normal bundle $\mathcal{O}(-1) \oplus
\mathcal{O}(-1)$. Moreover, in degree $11$, there are no singular,
reduced, and irreducible rational
curves, nor any reduced, reducible, and connected curves with rational
components on $F$.
\end{abstract}

\section* {Introduction}
Twenty-five years ago, H. Clemens conjectured that {\it the number of smooth
rational curves of fixed degree on a general quintic threefold in
$\mathbb{P}^4$ is finite}. Clemens' conjecture has attracted substantial interest in connection with mirror symmetry for the quintic. It has been verified for rational curves of degree at most 10 \cite{K, Klei, Cot}. In each case, the authors have proved a stronger statement; namely, that the incidence scheme
\vspace{-.25cm}
\begin{equation}
\Phi_d:=\{
(C,F)\mid\text{$F$ a quintic containing $C$}\}
\notag 
\vspace{-.25cm}
\end{equation}
is irreducible.

Proving that $\Phi_d$ is irreducible, in turn, is a two-step process, which involves:
\beg{enumerate}
\item[(i)] Establishing uniform bounds on the dimensions of the fibers of the projection of $\Phi_d$ to a suitable parameter space $M_d$ of degree-$d$ rational curves in $\mb{P}^4$.
\item[(ii)] Establishing that the locus inside $M_d$ comprising curves with large-dimensional fibers is of large codimension.
\end{enumerate}

In the present paper, we successfully resolve items (i) and (ii) when $d=11$. More precisely, we prove:
\begin{thm}\label{1.1}
The incidence
scheme $\Phi_d$ of smooth rational curves of degree $d$ at most $11$ on quintic
hypersurfaces $F \subset \mathbb{P}^4$ is irreducible. Moreover, any smooth curve $C$ lying on
a general quintic $F$ is embedded with normal bundle $\mathcal{O}_C(-1) \oplus
\mathcal{O}_C(-1)$. Furthermore, there are no rational and
singular, reduced, and irreducible curves of degree at most $11$ lying on
a general quintic in $\mathbb{P}^4$, other than
the six-nodal plane quintics.
\end{thm}

In \cite{Cot}, we proved the analogous result for rational curves of degree at most 10. To do so,
we introduced a new technique, based on a combinatorial analysis of the {\it generic initial ideals} which result from degenerations of rational curves,
to address the dimension-bounding issue (i) above. The generic initial ideal, or gin, is useful because it has a simple combinatorial structure; in particular, its cohomology is easily computable by hand.

In the current work we extend the methods of \cite{Cot} to handle
rational curves of degree $11$. In degrees at most 10, the resolution to item (ii) above follows easily from the algorithm we give for computing the cohomological invariants of a given gin. For degrees $d \geq 11$, item (ii) is significantly more delicate, because of certain monomial ideals corresponding
to curves lying on a large number of linearly independent hypersurfaces of low degree. 

{\fl \bf NB:} Johnsen and Kleiman \cite{Klei2} showed that $\Phi_d$ is reducible whenever $d \geq 12$, because of extraneous components arising from rational curves that lie on surfaces of low degree. Such components do not dominate the 125-dimensional projective space of quintic threefolds, however, so Clemens' conjecture is no less likely to be true.

\beg{thm}\label{coh} (Classification and cohomology of ideal sheaves; see Theorems~\ref{classif_hyp_gins} and \ref{coh_bounds} below.)
Let $\Ga$ denote a generic hyperplane section of a nondegenerate rational curve of degree 11 in $\mb{P}^4$. There are ten possibilities for $\mbox{gin}(\mc{I}_{\Ga})$. Let $i:= h^1(\mc{I}_C(5))$ and let $\mbox{reg}(\mc{I}_C)$ denote the regularity of $\mc{I}_C$. We have
\[
g+i \leq f_{\Ga,\mbox{reg}(\mc{I}_C)}
\]
where $8 \leq f_{\Ga,\mbox{reg}(\mc{I}_C)} \leq 13$ is an explicitly computable integer.
\end{thm}

By a standard Riemann--Roch argument, establishing upper bounds on the dimensions $h^0(\mc{I}_C(5))$ of fibers of the incidence correspondence $\Phi_{11}$ is equivalent to  establishing upper bounds for $g+i$. More precisely, letting $r$ denote the dimension of the linear span of $C$, it suffices to show that curves with $g+i >0$ determine a sublocus $M^{r,g}_{11,i} \sub M_{11}$ of codimension at least $g+i$, which implies that $\Phi_{11}$ is the closure of the subspace of curves with $i=0$. Since the latter subspace is a fiber bundle over an over open subspace of $M_{11}$, it then follows that $\Phi_{11}$ is irreducible.

In order to estimate the codimension of $M^{r,g}_{11,i}$, we further distinguish between rational curves $C$ with low and high arithmetic genus. When $g(C)$ is small and $i$ is large, $C$ will necessarily lie on either a hyperquadric or a hypercubic.
Curves on hyperquadrics have very special geometry; the interesting cases correspond to curves that lie on many linearly independent hypercubics. We show how to obtain a codimension estimate for such curves by applying a liaison-theoretic analysis to their inclusions (and residuals) in complete intersections of three suitably chosen hypercubics. Our liaison analysis also relies extensively on the theory for gins we develop in \cite{Cot}.

\begin{thm}\label{1.2} (Curves with large cohomology, and low arithmetic genus.) Let $C \sub \mb{P}^4$ denote a rational curve of degree 11 and arithmetic genus $g$, and let $i:=h^1(\mc{I}_C(5))$. Assume $g+i > 2g$, and that $C$ is neither linearly degenerate, nor a curve contained on a hyperquadric. Then $C$ lies on a surface of degree $m \leq 8$, contained in the complete intersection of three hypercubics. Moreover, letting $M^{4,g}_{11}(m)$ denote the corresponding subspace inside $M_{11}$, we have 
\[
g+i < \mbox{cod}(M^{4,g}_{11}(m), M_{11})
\]
for every integer $m \leq 8$.
\end{thm}

When the arithmetic genus $g$ is large, $i$ is small; correspondingly, the argument applied in the preceding case does not apply. However, the corresponding curves will be highly singular, which enables us to enables us to effectively bound their codimension.

\beg{thm} (Curves with large cohomology, and high arithmetic genus.) Let $M^{4,g}_{11}$ denote the subspace of $M_{11}$ comprising parameterized maps whose images are nondegenerate and of arithmetic genus $g$. We have
\[
\mbox{cod}(M^{4,g}_{11}, M_{11}) \geq \min(2g,12).
\]
\end{thm}

To complete the proof of Theorem~\ref{1.1}, it remains to obtain bounds on the dimensions of spaces of linearly degenerate and reducible curves, respectively. Here our arguments follow closely those of \cite[Thms 2.3, 3.1]{Cot}.

\beg{thm} (Linearly degenerate curves.)
The reduced irreducible rational curves that verify
\[
g+i \geq 8+ \min(g,5),
\]
and whose images span hyperplanes, comprise a sublocus of $M_{11}$ of codimension greater than $g+i$.
\end{thm}

\beg{thm} (Reducible curves.) On a general quintic threefold in $\mb{P}^4$, there is no connected, reducible and reduced curve of degree at most 11 whose components are rational.
\end{thm}

Note that in both \cite{Cot} and this paper, while gins allow us to bound the ideal
sheaf cohomology of curves adequately for our purposes, the estimates we obtain could be significantly 
improved by understanding which Borel-fixed ideals {\it actually arise} as gins of irreducible curves. An appropriate generalization of 
the ``connectedness" theorems of \cite{GP} and \cite{AT} for (ideal sheaves of) points in 
$\mathbb{P}^{2}$ and irreducible curves in $\mathbb{P}^{3}$, respectively, would likely place much more stringent
restrictions on the class of Borel-fixed ideals arising as gins than those provided by our analysis. 

Another potentially fruitful strategy is to prove a {\it tropical geometric} analogue of Clemens' conjecture. This is the ultimate goal of an ongoing project of the author, in collaboration with T. Bogart and E. Brugall\'e \cite{BBC}. Note that G. Hana and T. Johnsen \cite{HJ} have applied our method to prove that no rational curves of degree $2 \leq d \leq 15$ exist on a general degree-7 hypersurface $F \sub \mb{P}^5$; conjecturally, no rational curves of degree $d \geq 2$ exist on $F$.

Finally, a word on terminology. In our context, a {\it general} hypersurface $F$ is one that belongs to the complement of a countable union of proper Zariski-closed subschemes of the projective space of hypersurfaces. Elsewhere in the literature, such a hypersurface might be called ``very general".
 
\vspace{-10pt}
\section*{Acknowledgements.} I am grateful to I. Coskun, J. Harris, A. Iarrobino, M. Roth, and D. Smyth, for valuable conversations. I am especially grateful to S. Kleiman,
who has offered continual input, encouragement, and patience.

\tableofcontents

\section{Generic initial ideals}\label{sec-one}

\subsection{First properties}\label{firstproperties}

Given any homogeneous ideal $I$ in $(n+1)$ variables $x_i, 0 \leq i \leq n$,
together with a
choice of
partial order $<$ on the monomials $m(x_i)$, the {\it initial ideal} of $I$ with respect to $<$ 
is, by definition, generated by leading terms of elements in $I$. For any subscheme $C \sub \mathbb{P}^n$, upper-semicontinuity of cohomology implies that the regularity and values $h^j$
of the ideal sheaf $\mathcal{I}_C$ are majorized by those of the initial ideal sheaf
$\mbox{in}(\mathcal{I}_C)$ with respect to any partial order on the
monomials of $\mathbb{P}^n$. On the other hand, Macaulay's theorem establishes that the Hilbert
functions of $\mathcal{I}_C$ and $\mbox{in}(\mathcal{I}_C)$ agree. Moreover, if
we replace $C$ by
a {\it general} $\mbox{PGL}(n+1)$-translate of $C$, then
\begin{equation*}
\mbox{reg}(\mathcal{I}_C) = \mbox{reg}(\mbox{in}(\mathcal{I}_C)),
\end{equation*}
and $\mbox{in}(\mathcal{I}_C)$ is called the {\it generic initial
  ideal}
of $C$, written $\mbox{gin}(\mathcal{I}_C$). 

There is a unique generic initial ideal associated to any
  ideal. Further, the following three statements are equivalent.
\begin{enumerate}
\item The monomial ideal $\mbox{gin }\mathcal{I}_C$ is saturated, in
  the ideal-theoretic sense. In other words, for all $m \in (x_0,\dots,x_n)$,
\vspace{-.2cm}
\begin{equation*}
g \cdot m \in  \mbox{gin
  }\mathcal{I}_C \Rightarrow g \in \mbox{gin
  }\mathcal{I}_C.
\vspace{-.25cm}
\end{equation*}
\item The homogeneous ideal $I_C$
  is saturated.
\item No minimal generator of $\mbox{gin
  }\mathcal{I}_C \subset \mathbb{C}[x_0,\dots,x_n]$ is divisible by $x_n$.
\end{enumerate}
For proofs, see \cite[Thm.~2.30]{Gr}.

{\fl Now} let $I \subset \mathbb{C}[x_0,\dots,x_n]$ be any ideal that
is {\it Borel-fixed}, i.e., fixed under the action of upper triangular
matrices $\mathcal{T} \subset \mbox{PGL}(n+1)$; then $I$ is
saturated if and only if none of its minimal generators is divisible by $x_n$.
Hereafter, we work exclusively
with {\it saturated} generic initial ideals. 

A crucial fact is that {\it the Castelnuovo-Mumford regularity of any saturated ideal sheaf $\mc{I}$ equals that of its associated generic initial ideal sheaf, which in turn equals the maximum degree of any minimal generator of $\mbox{gin}(I)$.} This fact will be used in subsequent sections without comment.

\subsection{Hyperplane sections}\label{hypsections}
Our basic procedure for obtaining upper bounds on $h^1(\mc{I}_C(5))$, as developed in detail in \cite[Sect. $1$]{Cot}, is in two steps.

\beg{itemize}
\item[(1)] Determine possible monomial ideals arising as gins associated to general hyperplane sections of rational curves $C$.
\item[(2)] Determine possible curve gins corresponding to each hyperplane section gin, and compute the corresponding values of $h^1(\mc{I}_C(5))$.
\end{itemize}

{\fl In} this section, we prove the following result, which resolves item (1) above. 

\beg{thm}\label{classif_hyp_gins}
Let $I$ denote the gin of a general hyperplane section of an irreducible, nondegenerate rational curve of degree 11 in $\mb{P}^4$. There are ten possibilities for $I$.
\beg{enumerate}
\item $I =\mbox{Borel}(x_2^4,x_1x_2^2)$
\item $I= \mbox{Borel}(x_2^4,x_1^2x_2,x_0^2)$
\item $I= \mbox{Borel}(x_2^5,x_1x_2^2,x_0^2)$
\item $I= \mbox{Borel}(x_2^4,x_1^3,x_0x_2^2,x_0x_1)$
\item $I=\mbox{Borel}(x_2^4,x_1^2x_2,x_0x_2^3,x_0x_1)$
\item $I=\mbox{Borel}(x_2^5,x_1x_2^3,x_1^2x_2,x_0x_2^2,x_0x_1)$
\item $I=\mbox{Borel}(x_2^5,x_1x_2^2,x_0x_2^3,x_0x_1)$
\item $I = \mbox{Borel}(x_2^4,x_0x_2)$
\item $I = \mbox{Borel}(x_2^5,x_1x_2^3,x_1^3,x_0x_2)$
\item $I=\mbox{Borel}(x_2^5,x_1^2x_2,x_0x_2)$ 
\end{enumerate}
\end{thm}

{\fl Here} $\mbox{Borel}(J)$ denotes the smallest
Borel-fixed ideal containing an ideal $J$.

{\fl Before} giving the proof, we recall the basic properties of hyperplane gins that we shall need.
 
{\fl Each} hyperplane gin is minimally generated by monomials in $x_0,x_1,x_2$ of a set of homogeneous variables $x_i, 0 \leq i \leq 4$ for $\mb{P}^4$, and the corresponding set of minimal generators that may be represented by a tree (i.e., a directed graph without cycles), as follows. 

{\fl Fix} an alphabet
$\mathcal{A}:=\{\emptyset,x_0,x_1,x_2\}$, and consider the set of all
trees with root vertices $\emptyset$ and all other vertices labeled by either
$x_0,x_1,\text{ or }x_2$. Call
terminal vertices {\it leaves}. Then the unique path from the root vertex
$\emptyset$ to any leaf labeled $x_{i_l}$ determines a sequence
of vertices $\emptyset, x_{i_1},x_{i_2},\dots,x_{i_l}$.  (We will only be concerned
with trees associated to nonzero ideals, so $\emptyset$ will never
arise as a leaf.) It's natural to
interpret the string $x_{i_1}x_{i_2} \dotsm x_{i_l}$ as a polynomial of degree
$l$. Likewise, if the unique path from the root vertex to a given
vertex $v$ involves $d$ edges, then we say that {\it $v$ has degree
  $d$}. A {\it rewriting rule} applied to a leaf $v$ is then a formal operation on $T$ that
at $v$
glues a certain number of new edges $e_i(v)$ terminating in vertices $v_i(v)$. The result is a new tree $T^{\prime}$
in which the vertices $v_i(v)$ are leaves of degree $d+1$. Say that a vertex $V_1$ {\it dominates} a vertex $V_2$ if it
is closer to the root vertex $\emptyset$. We also stipulate that, if
$V_i$ with label $x_i$ dominates $V_j$ with label $x_j$, then $i \leq
j$. 

{\fl For} each homogeneous monomial ideal $I$, the unique minimal generating
set of $I$
determines a tree $T(I)$ whose leaves correspond to minimal generators
of $I$. The assignment $I \mapsto T(I)$, moreover,
is unique. Accordingly, we treat rewriting rules interchangeably as
operations on either ideals or trees.

{\fl In} \cite[Lemmas 1.2.1, 1.2.2]{Cot}, we showed: 
\beg{fact}\label{F1.1}
The minimal generating set of every nondegenerate hyperplane gin
may be realized by applying the rewriting rules given in Table ~\ref{Table $1$}. These we call $\Lambda$-rules. Graphically speaking, applying the $\Lambda$-rule $X_k \mapsto Y_{k+1}$ alters a tree of minimal generators
for $I_k$ by gluing $\mbox{Card}(Y_{k+1})$ new edges onto the leaf
corresponding to $X_k$.
\end{fact} 
\begin{table}[section]
\caption{$\Lambda$-rules for nondegenerate curves in $\mathbb{P}^4$}\label{Table $1$}
\vspace{-0.3cm}
$$\begin{array}{ll}
\hline
1.\rule{0 pt}{13 pt}&x_0^e
\mapsto (x_0^{e+1},x_0^ex_1,x_0^ex_2),\\
\vspace{.1cm}
2.&x_0^ex_1^f \mapsto (x_0^ex_1^{f+1},x_0^ex_1^fx_2), \\
\vspace{.1cm}
3.&x_0^ex_1^fx_2^g \mapsto
x_0^ex_1^fx_2^{g+1}, \text{and an {\it initial rule}}\\
4.&\emptyset \mapsto
(x_0,x_1,x_2).
\vspace{-.3cm} 
\end{array}$$
\end{table} 

\beg{fact}\label{F1.2}
The number of nonleaf vertices in the tree defining a zero-dimensional subscheme of $\mb{P}^3$ is equal to the scheme's degree.
\end{fact}

\beg{fact}\label{1.3}
The ideal sheaf of a set of $d$ points in uniform position in $\mb{P}^n$ is $(\lceil \fr{d-1}{n} \rceil+1)$-regular \cite{Ba}.
\end{fact}

{\fl The} upshot of the preceding discussion is that any ideal $I$ defining the gin of some general hyperplane section of a nondegenerate, irreducible rational curve $C$ is subject to significant numerical restrictions.

\beg{proof}[Proof of Theorem~\ref{classif_hyp_gins}]
Fact~\ref{1.3} above establishes that $I$ is 5-regular, so minimally generated by monomials $p(x_0,x_1,x_2)$ of degree at most $\leq 5$. Moreover, $T(I)$ has exactly 11 nonleaf vertices. As $I$ is Borel-fixed, it follows that the minimal generating set of $I$ contains a monomial of degree at least 4. 

Finally, $I$ cannot have more than $3$ quadratic generators. For, say $\Lambda \subset Q_1 \cap Q_2 \cap Q_3$, where the $Q_i$ are linearly independent quadrics in $\mathbb{P}^3$. By Bezout's theorem, $Q_1 \cap Q_2 \cap Q_3$ is not a complete intersection, because it contains a zero-dimensional subscheme of degree $11$. Thus $Q_1 \cap Q_2 \cap Q_3$ is a twisted cubic (or degeneration thereof), whose intersection with any fourth algebraically independent quadric containing $\Lambda$ is a zero-dimensional scheme of length at most 6, which is too small. 

The desired result now follows from an easy inspection. For example, the unique possibility when $I$ has no quadratic generators is the first ideal listed. Similarly, if $I$ has exactly one quadratic generator, there are two possibilities (the second and third ideals listed), which are distinguished by the fact that one is 4-regular while the other is not. The remaining ideals are obtained similarly.
\end{proof}

\subsection{Nondegenerate curve gins, and $h^1(\mc{I}_C(5))$}\label{curvegins}
In this subsection, we identify possible monomial ideals arising as gins of nondegenerate rational curves $C$ of degree $11$ and arithmetic genus $g$ in $\mb{P}^4$, and using the method of \cite[section $1.4$]{Cot}, we obtain bounds on $g+i$, where $i:=h^1(\mc{I}_C(5))$.
According to Fact~\ref{1.4} below, each $\mbox{gin}(\mc{I}_C)$ is obtained by applying a sequence of rewriting rules to $\mbox{gin}(\mc{I}_{\Ga})$. Accordingly, the mapping space $M_{11}$ is naturally stratified by the set of hyperplane gins.

\beg{thm}\label{coh_bounds} Let $C \sub \mb{P}^4$ be a nondegenerate rational curve of degree 11, with generic hyperplane section $\Ga$. The invariant $g+i$ of $C$ satisfies the following upper bounds.
\beg{enumerate}
\item If $\mbox{gin}(\mc{I}_{\Ga})=\mbox{Borel}(x_2^4,x_1x_2^2)$, then $g+i \leq 8$.
\item If $\mbox{gin}(\mc{I}_{\Ga})=\mbox{Borel}(x_2^4,x_1^2x_2,x_0^2)$, then $g+i \leq 9$.
\item If $\mbox{gin}(\mc{I}_{\Ga})$ is one of the following ideals, then $g+i \leq 10$.
\beg{itemize}
\item $\mbox{Borel}(x_2^5,x_1x_2^2,x_0^2)$
\item $\mbox{Borel}(x_0x_1,x_0x_2^3,x_1^2x_2,x_2^4)$
\item $\mbox{Borel}(x_0x_1,x_0x_2^2,x_1^3,x_2^4)$
\end{itemize}
\item If $\mbox{gin}(\mc{I}_{\Ga})$ is one of the following ideals, then $g+i \leq 11$.
\beg{itemize}
\item $\mbox{Borel}(x_0x_2,x_2^4)$
\item $\mbox{Borel}(x_0x_1,x_0x_2^2,x_1^2x_2,x_1x_2^3,x_2^5)$
\item $\mbox{Borel}(x_0x_1,x_0x_2^3,x_1x_2^2,x_2^5)$
\end{itemize}
\item If $\mbox{gin}(\mc{I}_{\Ga})=\mbox{gin}(\mc{I}_{\Ga})=\mbox{Borel}(x_2^5,x_1^2x_2,x_0x_2)$, then $g+i \leq 12$.
\item If $\mbox{gin}(\mc{I}_{\Ga}= \mbox{Borel}(x_2^5,x_1x_2^3,x_1^3,x_0x_2)$, then $g+i \leq 13$.
\end{enumerate}
\end{thm}


{\fl As} was the case with Theorem~\ref{classif_hyp_gins}, the proof of Theorem~\ref{coh_bounds} is nearly immediate once the basic structure of the ideals in question has been explained. For this purpose, let $x_0,\dots,x_4$ be homogeneous coordinates on $\mb{P}^4$. The saturated generic initial ideal $\mbox{gin}(\mc{I}_C)$ is a Borel-fixed monomial ideal, none of whose minimal generators is divisible by $x_4$. To each minimal generating set of $\mbox{gin}(\mc{I}_C)$ we can associate a tree $T(\mbox{gin}(\mc{I}_C))$. Then, as explained in \cite[section $1.2$]{Cot}, we have:

\beg{fact}\label{1.4}
The tree corresponding to any
generic initial ideal defining a nondegenerate subscheme $X \subset \mathbb{P}^4$ of dimension $1$ is
obtained from the $\mbox{gin}$ of a general hyperplane section of
$X$ by applying a sequence of rewriting rules given in \ref{Table $2$}, that we denote by $C$-rules.
\end{fact}
\beg{fact}\label{1.5}
Given any hyperplane gin $I$, let $I_{C_{\Gamma}}$ be its extension to $\mb{C}[x_0,\dots,x_4]$. The scheme defined by $C_{\Gamma}$ is a cone with vertex $(0,0,0,0,1)$ over the zero-dimensional scheme defined by $I$. Let $m$ be any positive integer such that $\mc{I}_{C_{\Gamma}}$ is $m$-regular. Then 
\[
g(C_{\Gamma})=11m+1-\binom{m+4}{4}+h^0(\mc{I}_{C_{\Gamma}}(m)).
\]
\end{fact}
\beg{fact}\label{1.6}
Let $I \subset \mathbb{C}[x_0,\dots,x_4]$ be a Borel-fixed monomial ideal
  of dimension $1$, and $I^{\prime}$ the ideal obtained from $I$ by applying a
  single $C$-rule. Then $g(I^{\prime})=g(I)-1$, where $g(I)$ denotes the arithmetic genus of the scheme defined by $I$.
\end{fact}
\beg{fact}\label{1.7}
The number of $C$-rewritings applied to yield $\mbox{gin}(\mc{I}_C)$ from a given hyperplane gin $I_{\Gamma}$ equals $g(C_{\Gamma})-g(C)$.
\end{fact}
\beg{fact}\label{1.8}
The number of vertices in the tree representing $\mbox{gin}(\mathcal{I}_C)$ dominating vertices of degree greater than
  5 (equivalently, the number of rewritings applied to vertices of
  degree 6 or greater) equals $h^1(\mc{I}_C(5))$.
\end{fact}
\beg{fact}\label{1.9} For all nondegenerate rational curves $C \sub \mb{P}^4$, we have
\[ g+ h^1(\mc{I}_C(5)) \leq g(C_{\Gamma}) \]
\end{fact}
\beg{fact}\label{1.10}
The number of vertices in the tree representing $\mbox{gin}(\mathcal{I}_{\Ga})$ dominating vertices of degree greater than
  $t$ (equivalently, the number of $\La$-rewritings applied to vertices of
  degree $(t+1)$ or greater) equals $h^1(\mc{I}_{\Ga}(t))$, for all $t \geq 0$.
\end{fact}

\begin{table}
\caption{$C$-rules for nondegenerate irreducible curves}\label{Table $2$}
\vspace{-.3cm}
$$\begin{array}{ll}
\hline
\vspace{0.1cm}
1.&x_0^e
\mapsto x_0^e \cdot(x_0,x_1,x_2,x_3),\\
\vspace{.1cm}
2.&x_0^ex_1^f \mapsto x_0^ex_1^f \cdot(x_1,x_2,x_3), \text{ and }\\
\vspace{.1cm}
3.& x_0^ex_1^fx_2^g \mapsto
x_0^ex_1^fx_2^g \cdot(x_2,x_3) \notag \\
\vspace{.1cm}
4.& x_0^ex_1^fx_2^gx_3^h \mapsto
x_0^ex_1^fx_2^gx_3^{h+1} \notag
\vspace{-.3cm}
\end{array}$$

\vspace{10pt}
Note that the last line of Table 2 was accidentally omitted in the earlier paper \cite{Cot}.
\end{table}
 
 \beg{proof}[Proof of Theorem~\ref{coh_bounds}.]
As noted in Section~\ref{firstproperties}, the regularity of the sheaf determined by a saturated Borel-fixed monomial ideal is equal to the maximum degree of its minimal generators. On the other hand, for any given $\mbox{gin}(\mc{I}_{\Ga})$ of regularity $m$, Fact~\ref{1.5} and Fact~\ref{1.9} together yield an explicit upper bound on $g+i$ that depends on $m$ alone, namely
\beg{equation}\label{g+ibound}
g+i \leq 11m+1 - \binom{m+4}{4}+ h^0(\mc{I}_{C_{\Ga}}(m)).
\end{equation}
Here $i$, as usual, denotes $h^1(\mc{I}_C(5))$, where $C$ is any nondegenerate degree-11 rational curve in $\mb{P}^4$ with generic hyperplane section $\Ga$, while $C_{\Ga}$ is the cone with vertex $(0,0,0,0,1)$ over the zero-dimensional subscheme of $\mb{P}^3$ defined by $\mbox{gin}(\mc{I}_{\Ga})$. Concretely, $h^0(\mc{I}_{C_{\Ga}}(m))$ computes the number of distinct monomials of degree $m$ in $x_0, \dots, x_4$ obtainable as $\mb{C}[x_0,\dots,x_4]$-linear combinations of minimal generators of $\mbox{gin}(\mc{I}_{\Ga})$. Computing the right-hand side of \eqref{g+ibound} for each choice of $\mbox{gin}(\mc{I}_{\Ga})$ yields the desired result.
 \end{proof}

\section{Irreducible curves}\label{sec-two}
In this section, we'll prove Theorem~\ref{1.1}, which extends
results of Johnsen, Kleiman, and the author to include curves of degree at most
11. See \cite[Thm. 3.1]{Klei} and \cite[Thm. 2.1]{Cot}.

\subsection{Initial reductions}
Let $M_{11}^{r,g}$ denote the affine space of parameterized
mappings of $\mathbb{P}^1$ into $\mathbb{P}^4$ that map birationally
onto images of degree $11$, of arithmetic genus $g$, and span an
$r$-dimensional projective space. Note that we do {\it not} mod out
by the $\mbox{PGL}(2)$-action on $\mathbb{P}^1$; nor do we projectivize. Let $\Phi_{11}^{r,g}$ denote the corresponding incidence locus in
$M_{11}^{r,g} \times \mathbb{P}^{125}$. Let $M_{11,i}^{r,g} \subset
M_{11}^{r,g}$ denote the subspace of morphisms with images $C$
satisfying
\begin{equation*}
h^1(\mathcal{I}_C(5))=i, 
\end{equation*}
and let $\Phi_{11,i}^{r,g} \subset \Phi_{11}^{r,g}$ denote
the pullback of $M_{11,i}^{r,g}$ under the canonical projection $\Phi_{11}^{r,g}
\rightarrow \mathbb{P}^{125}$.

As explained in \cite{Klei}, to
prove Theorem~\ref{1.1} it suffices to show that 
\begin{equation}\label{60-g-i}
\mbox{cod}(M_{11,i}^{r,g}, M_{11}) > g+i
\end{equation}
whenever $i$ is nonzero.

In order to simplify the notation, we let $i:=h^1(\mathcal{I}_C(5))$
hereafter.

We may also assume $r \geq 3$, since (as noted in \cite[Section 3]{Klei})
no planar rational curves of degree greater than 5 lie on a general quintic.

We'll need the following analogue of \cite[Lemma 3.4]{Klei}, whose proof we sketch in the final section; for full details see \cite{Cot2}.
\beg{thm}\label{singest}
We have
\begin{subequations}\label{E1.2}
\begin{align}
\dim M_{11}^{4,g} &\leq 60-\min(2g,12), \label{1.2a} \text{ and }\\
\dim M_{11}^{3,g} &\leq 52-\min(g,5)\label{1.2b}.
\end{align}
\end{subequations}
\end{thm}
The second estimate was already proved in \cite{Klei}. The attentive reader will note an apparent discrepancy of 4 between our upper bound for $\dim M_{11}^{3,g}$ and that given in \cite{Klei}; the reason our upper bound is 4 more than theirs is that we allow our spanning hyperplanes $H$ to vary, while Johnsen and Kleiman do not. This bit of notational confusion led to an error in \cite{Cot}, which is corrected in an erratum included at the end of \cite{Cot2}.

To establish the validity of \eqref{60-g-i}, we treat nondegenerate
curves and curves with 3-planar images separately. We first consider nondegenerate
curves in $\mathbb{P}^4$. Applying \eqref{1.2a} of Theorem~\ref{singest}, Theorem~\ref{1.1} for nondegenerate
curves follows from the following result.

\begin{thm}\label{1.1bis}
The nondegenerate reduced irreducible
rational curves verifying
\begin{equation*}
g+i \geq \min(2g,12)
\end{equation*}
define a sublocus of the parameter space $M_{11}$ of degree-$11$
rational maps of codimension greater than
$g+i$.
\end{thm}


\subsection{Special subloci associated to secants and hyperquadrics}
In this subsection, we obtain lower bounds on the codimensions of
certain special subloci of $M_{11}$. These estimates allow us conclude that certain curves $C$ corresponding to generic initial ideals with large $i$ never lie on a general quintic threefold, and therefore, may be ignored. 

\beg{defn}\label{nonproblematic}
Let $V \sub M_{11}$ denote a locally closed subvariety of rational curves $C$ of degree 11 with with arithmetic genus $g$ and $h^1(\mc{I}_C(5))=i$. Then $V$ is {\bf nonproblematic} if $V$ is of codimension greater than $g+i$.
Similarly, a curve $C \in M_{11}$ is nonproblematic if either
\beg{enumerate}
\item[(1)] $C$ belongs to a nonproblematic subvariety $V \sub M_{11}$, or 
\item[(2)] $g+i<\min(2g,12)$.
\end{enumerate}
\end{defn}

\begin{lem}\label{1a}
Degree-$11$ morphisms determining
nondegenerate degree-$11$ rational curves
with $9$-secant lines (resp., linearly degenerate degree-11 rational curves with 10-secant lines) have codimension at least $12$ (resp., 6) in $M^4_{11}$ (resp., in $M^3_{11}$).
\end{lem}

\begin{proof}
This follows from an easy dimension count. In fact, the proof of \cite[Lemma 2.2.1]{Cot} shows, more generally, that whenever $4 \leq d^{\pr} \leq 11$, those elements of $M^4_{11}$ (resp., in $M^3_{11}$) whose images have $d^{\pr}$-secant lines have codimension at least $(2d^{\pr}-6)$ (resp., $(d^{\pr}-4)$).
\end{proof}

\begin{lem}\label{1b}Degree-$11$ morphisms defining rational curves in reduced, irreducible hyperquadrics in $\mb{P}^4$ determine a locus of codimension
  at least 9 in 
$M_{11}$.
\end{lem}

\begin{proof}
This, too, follows from a dimension count. Those morphisms whose images lie in a {\it fixed} hyperquadric have codimension $2 \cdot 11+1=23$. Hyperquadrics determine a $\mb{P}^{14}$; allowing these to vary, we obtain the desired estimate. See the proof of \cite[Lemma 2.2.2]{Cot}, which also treats the case of singular hyperquadrics.
\end{proof}

\subsection{RTB stratification of the mapping space}
Assume $n \geq 3$, and let $f\:\mathbb{P}^1 \rightarrow \mathbb{P}^n$ be a map with image
$C$ of degree $d$. Note that the $(-1)$-twist of the restricted
tangent bundle, $f^* (T_{\mathbb{P}^n}(-1))$, has degree $d$ and rank $n$ over
$\mathbb{P}^1$, say with splitting 
\begin{equation*}
f^* (T_{\mathbb{P}^n}(-1))=
\bigoplus_{i=1}^n \mathcal{O}_{\mathbb{P}^1}(a_i). 
\end{equation*}
A result of Verdier's (see \cite{V} and \cite{R})
establishes that the scheme of morphisms $\mathbb{P}^1 \rightarrow
\mathbb{P}^n$ of degree $d$ corresponding to a particular splitting type
$(a_1,\dots,a_n)$ with $a_1 \geq \dots \geq a_n$ is
irreducible of the expected codimension
\begin{equation*}
\sum_{i\neq j} \max\{0,a_i-a_j-1\}.
\end{equation*}
\beg{defn}
Let the {\bf RTB stratification} of $M_d^n$ denote the stratification of $M_d^n$ according
to the splitting type of the restricted tangent bundle.
\end{defn}

When $n=4$ and $d=11$, every special splitting stratum has codimension 2 or more, and (4,3,2,2) is the splitting type associated to the unique stratum of codimension 2. More special strata have codimension at least 6.

\subsection{Beginning of the proof of Theorem~\ref{1.1bis}}\label{beginofproof}
We will deduce Theorem~\ref{1.1bis} as a consequence of three lemmas, the first of which we prove in this section.

\beg{lem}\label{lem 2.3}
Let $C \sub \mb{P}^4$ denote a nondegenerate rational curve of degree 11 such that $g+i > 2g$. Then either $C$ is nonproblematic, or $g \leq 2$ and $h^0(\mc{I}_C(3)) \geq 5$.
\end{lem}

\beg{proof}[Proof of Lemma~\ref{lem 2.3}]
We argue separately for each possible $\mbox{gin}(\mathcal{I}_{\Ga})$. As a matter of convenience, we group possible hyperplane gins according to their number of quadratic minimal generators. In each case, we consider the possible $\mbox{gin}(\mc{I}_C)$ corresponding to the given choice of $\Ga$. Whenever such a generic initial ideal corresponds to a nonproblematic curve, it may be discarded.

{\fl \bf Case 1.} {\it Say $\mbox{gin}(\mc{I}_{\Gamma})=\mbox{Borel}(x_2^4,x_1x_2^2)$.}

Theorem~\ref{coh_bounds} implies $g+i \leq 8$ for every irreducible nondegenerate $C$ with generic hyperplane section $\Ga$.
By ~\cite[Thm. 3.1]{GLP}, $\mc{I}_{C}$ is
$8$-regular, unless $C$ admits a $9$-secant line. On the other hand, by Lemma~\ref{1a}, rational degree-$11$ curves with 9-secant lines determine a subvariety of $M_{11}$ of codimension at least 12. Consequently, 8-irregular curves are nonproblematic.

Assume now that $C$ is $8$-regular. If $g \geq 5$, then since $g+i \leq 8$ we have $g+i < \min(2g,12)$, so $C$ is nonproblematic. On the other hand, if $g \leq 5$, then in fact $g+i \leq 5$. To see why, note that $\mbox{gin}(\mc{I}_C)$ is obtained from $\mbox{Borel}(x_2^4,x_1x_2^2)$ by applying at most eight $C$-rewriting rules. The first three of these automatically occur in degrees less than 6 since $\mbox{Borel}(x_2^4,x_1x_2^2)$ is minimally generated in degrees at most 4.

Next, assume that the sheaf $f^*(T_{\mb{P}^4}(-1))$ associated to the map $f\: \mb{P}^1 \rt \mb{P}^4$ defining $C$ has splitting type $(3,3,3,2)$ or $(4,3,2,2)$, since otherwise $C$ belongs to a codimension-6 subvariety of $M_{11}$ and is, therefore, nonproblematic. Then \cite[Proposition 1.2]{GLP} implies that $\mc{I}_C$ is $7$-regular. The same proposition implies that if $f^*(T_{\mb{P}^4}(-1))$ has splitting type $(3,3,3,2)$, then $C$ is 6-regular. So we may assume that $f^*(T_{\mb{P}^4}(-1))$ has splitting type $(4,3,2,2)$, and that $\mc{I}_C$, and consequently $\mbox{gin}(\mc{I}_C)$, is 7-regular.

Next, note that any $C$ for which $\mc{I}_C$ is 7-regular with hyperplane gin $\mbox{Borel}(x_2^4,x_1x_2^2)$ verifies $i \leq 3$. Namely, $\mbox{gin}(\mc{I}_C)$ is obtained from $\mbox{Borel}(x_2^4,x_1x_2^2)$ by applying at most eight $C$-rewriting rules, and because $C$ is $7$-regular, all of these are in degrees less than 7. Without loss of generality, assume that the first rewriting applied is $x_2^4 \mapsto (x_2^5,x_2^4x_3)$, and that the ideal obtained after six rewritings is $\mbox{Borel}(x_2^4x_3^3,x_1x_2^2)$. Any two further rewritings resulting in a $7$-regular ideal necessarily occur in degrees less than 6, so $i \leq 3$, as desired.

Since $i \leq 3$, we deduce that $g+i < \min(2g,12)$ (and, consequently, $C$ is nonproblematic) whenever $g \geq 4$. Moreover, if $g=3$, then in fact $i \leq 2$, in which case $C$ is nonproblematic. So assume $g \leq 2$. If $g=2$, we observe the following dichotomy: either
\[
g+i < 2g, \text{ or } h^0(\mc{I}_C(3)) \geq 7.
\]
In other words, if $g=2$, then either $C$ is nonproblematic, or $h^0(\mc{I}_C(3)) \geq 7$.

Similarly, if $C$ is of genus $0$ or $1$, then $C$ obeys the following basic trichotomy: either
\[
g+i< 2, \hspace{5pt} g+i < 2g, \text{ or } h^0(\mc{I}_C(3)) \geq 7
\]
unless $g=1$ and
\[
\beg{split}
\mbox{gin}(\mc{I}_C)&=\mbox{Borel}(x_2^4x_3^3,x_2^5x_3,x_1x_2^2x_3,x_0x_2^2), \text{ or}\\
\mbox{gin}(\mc{I}_C)&=\mbox{Borel}(x_2^4x_3^3,x_2^5x_3,x_1x_2^2x_3,x_1^3,x_0x_2^2x_3, x_0x_1x_2),
\end{split}
\]
or $g=0$ and
\[
\beg{split}
\mbox{gin}(\mc{I}_C)&=\mbox{Borel}(x_2^4x_3^3,x_2^6,x_1x_2^2x_3,x_0x_2^2), \text{ or} \\
\mbox{gin}(\mc{I}_C)&=\mbox{Borel}(x_2^4x_3^3,x_2^6,x_1x_2^2x_3,x_1^3,x_0x_2^2x_3).
\end{split}
\]
Each of the latter four ideals is such that $g+i=2$. Using Macaulay2 \cite{GS}, on the other hand, one checks that the {\it generic} element $C$ of the $(4,3,2,2)$ stratum is 4-regular, and therefore verifies $i=0$. See Section~\ref{sec-five} for a proof. So any curve with gin equal to any one of the four ideals listed above belongs to a closed subspace of codimension at least 3 in $M_{11}$, and is, therefore, nonproblematic. 
Lemma~\ref{2.3} follows immediately in this case.

{\fl \bf Case 2: $\mbox{gin}(\mc{I}_{\Ga})$ has exactly one quadratic
generator.} \\
According to Section~\ref{hypsections}, we have
\beg{enumerate}
\item $\mbox{gin}(\mc{I}_{\Ga})=\mbox{Borel}(x_2^4,x_1^2x_2,x_0^2)$,
or
\item $\mbox{gin}(\mc{I}_{\Ga})= \mbox{Borel}(x_2^5,x_1x_2^2,x_0^2)$.
\end{enumerate}
In the first subcase, Theorem~\ref{coh_bounds} yields $g+i \leq 9$. It now follows from Lemma~\ref{1a} (resp., Lemma~\ref{1b}) that 8-irregular curves (resp., curves that lie on hyperquadrics) with hyperplane section as in the first case are nonproblematic. Moreover, if $C$ is
8-regular but 7-irregular with a hyperplane section gin as in the first subcase, then
inspection yields $i \leq 5$, with equality if and
only if
\[
\mbox{gin}(\mc{I}_C)=\mbox{Borel}(x_0^2x_3,x_0x_2^2,x_1^2x_2,x_1x_2^3,x_2^6x_3,x_2^4x_3^4).
\]
On the other hand, by \cite[Prop. 1.2]{GLP}, we have $h^1(\mc{I}_C(6))=0$ whenever the corresponding (twisted) restricted tangent bundle $f^*T_{\mathbb{P}^4}(-1)$ has splitting type $(4,3,2,2)$, while the
next-most generic splitting stratum $(4,4,2,1)$ has codimension 6. Whence, we
may assume $C$ is 7-regular. Our upper bound on $i$ improves, accordingly, to $i \leq 3$. Arguing as in Case 1, we deduce that either $C$ is nonproblematic, or $g(C) \leq 2$, and $C$ is contained in seven linearly independent hypercubics.

{\fl Similarly}, in the second subcase, we have $g+i \leq 10$. As before, we may assume any $C$ with the corresponding
hyperplane section gin is 8-regular and lies on no hyperquadric. On
the other hand, if $C$ is 8-regular but 7-irregular, then inspection yields $i \leq
6$, with equality if and only if $h^0(\mc{I}_C(3))=9$. However, if $C$ is 7-irregular, then its corresponding restricted tangent bundle
$f^*T_{\mathbb{P}^4}(-1)$ has splitting type of codimension 6 or more, by \cite[Prop. 1.2]{GLP}.
On the other hand, the codimension-6 splitting stratum $(4,4,2,1)$ is irreducible, and one can check with Macaulay2 (see Section~\ref{sec-five}) that a generic element of
the latter splitting stratum is 6-regular. Consequently, we may
further assume $C$ to be 7-regular; our bound on $i$ then
improves to $i \leq 3$. It now follows as in Case 1 that if $g \geq 3$, then $C$ is nonproblematic. Further, if $g=0$, then either $C$ is nonproblematic, or else $i=3$, and $C$ is contained in seven linearly independent hypercubics. However, if $g=1$ and $i=2$, then we can only deduce that $C$ is contained in {\it five} linearly independent hypercubics; e.g., we might have
\[
\mbox{gin}(\mc{I}_C)= \mbox{Borel}(x_0^2x_3,x_0x_1^2,x_1x_2^2x_3,x_2^5x_3^2).
\]
Similarly, if $g=2$ and $i=2$, then $C$ might be contained in {\it six} linearly independent hypercubics.

{\fl \bf Case 3: $\mbox{gin}(\mc{I}_{\Ga})$ has exactly two quadratic
generators.} \\ If  $\mbox{gin}(\mc{I}_{\Ga})$ is 4-regular,
then either
\beg{enumerate} 
\item $\mbox{gin}(\mc{I}_{\Ga})=\mbox{Borel
}(x_0x_1,x_0x_2^3,x_1^2x_2,x_2^4)$, or
\item $\mbox{gin}(\mc{I}_{\Ga})=\mbox{Borel
}(x_0x_1,x_0x_2^2,x_1^3,x_2^4)$;
\end{enumerate}
in either subcase, we have $g+i \leq 10$, by Theorem~\ref{coh_bounds}.

Once more, curves lying in hyperquadrics and 8-irregular
curves are nonproblematic, and may be ignored. Exactly as in Case 1, we deduce that either $C$ is nonproblematic, or $g(C) \leq 2$, and $C$ is contained in seven linearly independent hypercubics.

On the other hand, if $\mbox{gin}(\mc{I}_{\Ga})$ is 4-irregular, then either
\beg{enumerate}
\item $\mbox{gin}(\mc{I}_{\Ga})=\mbox{Borel
}(x_0x_1,x_0x_2^2,x_1^2x_2,x_1x_2^3,x_2^5)$,
or
\item $\mbox{gin}(\mc{I}_{\Ga})=\mbox{Borel
}(x_0x_1,x_0x_2^3,x_1x_2^2,x_2^5)$;
\end{enumerate}
in either subcase, we have $g+i \leq 11$, by Theorem~\ref{coh_bounds}.

As before, curves lying in hyperquadrics and 8-irregular
curves are nonproblematic, and so may be ignored. As in Case 2, we find that $C$ is either nonproblematic, or $g(C) \leq 2$ and $C$ is contained in seven linearly independent hypercubics, with 2 exceptions. Namely, $C$ might be contained in five linearly independent hypercubics, with $g=1, i=2$, or $C$ might be contained in six linearly independent hypercubics, with $g=2, i=2$.  


{\fl \bf Case 4: $\mbox{gin}(\mc{I}_{\Ga})$ has exactly three quadratic
generators.} \\ If $\mbox{gin}(\mc{I}_{\Ga})$ is 4-regular, we have
\[
\mbox{gin}(\mc{I}_{\Ga})= \mbox{Borel}(x_0x_2,x_2^4).
\]
Here $g+i \leq 11$, by Theorem~\ref{coh_bounds}. As usual, curves lying in hyperquadrics and 8-irregular
curves are nonproblematic, and may be ignored. Exactly as in Case 1, we deduce that either $C$ is nonproblematic, or $g(C) \leq 2$ and $C$ is contained in seven linearly independent hypercubics.

On the other hand, If $\mbox{gin}(\mc{I}_{\Ga})$ is 4-irregular, then
\[
\mbox{gin}(\mc{I}_{\Ga})= \mbox{Borel}(x_2^5,x_1x_2^3,x_1^3,x_0x_2), \text{ or } \mbox{gin}(\mc{I}_{\Ga})=\mbox{Borel}(x_2^5,x_1^2x_2,x_0x_2).
\]
Here $g+i \leq 12 \text{ (resp., }g+i \leq 13)$ by Theorem~\ref{coh_bounds} for the first (resp., second) hyperplane gin. Arguing much as before, we deduce that either $C$ is nonproblematic, or $g(C) \leq 2$ and $C$ is contained in at least five linearly independent hypercubics. 


\end{proof}
\subsection{Special geometry of rational curves with $h^0(\mc{I}_C(3)) \geq 5$}
Our next task is to understand the geometry of the linear system $\mathcal{L}$ of hypercubics
containing a nondegenerate degree-11 rational curve $C$ with $h^0(\mc{I}_C(3)) \geq 5$. 

{\fl To} this end, choose a basis of linearly independent hypercubics
$\mathcal{B}=\{X_1, \dots, X_5\}$ for $\mc{L}$. There are two basic possibilities to consider. Namely:
\begin{enumerate}
\item Every triple intersection $X_i \cap X_j \cap X_k, 1 \leq i,j,k \leq 5$ contains an irreducible
surface component $S_{i,j,k}$, or 
\item For some triple $(i,j,k)$, $X_i \cap X_j \cap X_k$ is a 1-dimensional
complete intersection.
\end{enumerate}

{\fl In this section, we treat the first possibility. For} simplicity, choose $(i,j,k)=(1,2,3)$ and
set $S:=S_{1,2,3}$, $X:= X_1 \cap X_2$. Let $S^{\prime}$ denote the
residual of $S$ in $X$.

Since $C$ lies on $X$ and is irreducible, $C$ necessarily lies
on $S$ or $S^{\prime}$, and since $S$ lies in $X_1 \cap
X_2 \cap X_3$, it follows that $C \subset S$. Moreover, Max Noether's
$AF+BG$ theorem \cite[p. 703]{GH} implies that $X_1 \cap X_2 \cap X_3$ is strictly contained in $X_1 \cap X_2$, so $S$ has degree at most 8. Now let $X=X_1
\cap X_2= S \cup S^{\prime}$, where $S$ and $S^{\prime}$ are surfaces
without any common components. Thus, $\mbox{deg }S+\mbox{deg
}S^{\prime}=9$.

\beg{lem}\label{surface case}
Assume that $C \sub \mb{P}^4$ is a nondegenerate rational curve of degree 11, that $g(C) \leq 2$, $C$ is contained in no hyperquadric, $h^0(\mc{I}_C(3)) \geq 5$, and that, moreover, $C$ is contained in a {\it subsurface} of the nontransversal intersection of 3 linearly independent hypercubics. Then $C$ is nonproblematic.
\end{lem}

\beg{proof}[Proof of Lemma~\ref{surface case}.]
{\fl We} now argue case-by-case, according to the degree of $S$, beginning with
$\deg(S)=4$. Here $\deg(S) \geq 3$ because $S$
contains the nondegenerate curve $C$. Moreover, if $\deg(S)=3$, then $S$ is either a nondegenerate cubic scroll
  in $\mathbb{P}^4$  or a cone over a twisted cubic in $\mathbb{P}^3
  \subset \mathbb{P}^4$, either of which lies on a hyperquadric. Replacing $S$ by any maximal-degree component of the corresponding maximal reduced subscheme $S^{\mbox{red}}$, we may further assume $S$ is irreducible and reduced.

\begin{enumerate}
\item {\bf $\mbox{deg }S=4$.} By \cite[Thm 2.1]{BrSc}, $S$ is the projection from $\mathbb{P}^5$ of either a rational normal scroll or
  the Veronese surface.
Note that every curve $C$ lying on $S$ arises as the projection of some irreducible
rational curve on either the
quartic scroll or the Veronese surface in $\mathbb{P}^5$.

{\bf Quartic scroll case.} Recall (see \cite[522-523]{GH}) there are two types of quartic scrolls
$S$. The first, $S_{2,2}$, is the
image of $\mb{F}_0:=\mathbb{P}^1 \times \mathbb{P}^1$ under the birational
map $\phi_{2,2}$
associated to the line bundle $\mc{O}_{\mb{F}_0}(e+2f)$. Here $e$ and $f$
are the fiber classes on $\mathbb{P}^1 \times \mathbb{P}^1$, so $e^2=f^2=0$ and $e \cdot
f=1$.

We check next that
\begin{equation}\label{h^0ofF_0}
h^0(\mc{O}_{\mb{F}_0}(e+2f))=6. 
\end{equation}
To this end, note that the canonical class on $\mb{F}_0$ is
\begin{equation*}
K_{\mb{F}_0}=-2e-2f.
\end{equation*}

Riemann--Roch yields
\begin{equation*}
\chi(\mc{O}_{\mb{F}_0}(e+2f))=6.
\end{equation*}

On the other hand, Serre duality yields
\begin{equation*}
h^2(\mc{O}_{\mb{F}_0}(e+2f))=h^0(K_{\mb{F}_0}(-e-2f))=0,
\end{equation*}
since the line bundle $K_{\mb{F}_0}(-e-2f)$ has negative degree. 

To deduce \eqref{h^0ofF_0}, it remains to
show that $h^1(\mc{O}_{\mb{F}_0}(e+2f))=0$. For that purpose, note that for all
$m \in \mathbb{Z}$ there is an
exact sequence
\begin{equation*}
0 \rightarrow \mc{O}_{\mb{F}_0}(e+mf) \rightarrow \mc{O}_{\mb{F}_0}(e+(m+1)f) \rightarrow \mc{O}_f \cong
\mc{O}_{\mathbb{P}^1} \rightarrow 0.
\end{equation*}
Here $H^1(\mc{O}_{\mb{F}_0})=0$. It follows that $H^1(\mc{O}_{\mb{F}_0}(e+mf))=0$ for all $m \geq 0$.

From \eqref{h^0ofF_0}, we conclude that the maps $\phi_{2,2}: \mb{F}_0
\rightarrow \mathbb{P}^5$, each of which is given by 6 sections of
$\mc{O}_{\mb{F}_0}(e+2f)$, determine a 36-dimensional family. Similarly, the
projected images $\phi^{\prime}_{2,2}: \mb{F}_0 \rightarrow \mathbb{P}^4$
are given by 5 sections of $\mc{O}_{\mb{F}_0}(e+2f)$, so they comprise a
30-dimensional family. Quotienting by the action of the 6-dimensional automorphism group of $\mb{F}_0$, it follows that
projections of quartic scrolls of type $S_{2,2}$ comprise a
24-dimensional family. 

Next, we bound the dimension of the space of rational curves $C$ of degree 11 on a fixed scroll $S_{2,2}$. Here we assume $C$ is of genus $g \leq 2$, since $C$ is nonproblematic otherwise. 

Letting $[C]=ae+bf$, the adjunction formula, coupled with the fact that $\deg(C)=11$, now implies that
\[
ab-a-b+1=g, \text{ and } 2a+b=11
\]
where $0 \leq g \leq 2$. It follows immediately that $g=0$ and either $(a,b)=(1,5)$ or $(a,b)=(9,1)$.
In general, curves of type $(a,b)$ vary in a
family of dimension $a+b+2$. Therefore, those of types $(1,5)$ and $(9,1)$
determine 8- and 12-dimensional families, respectively. It follows that genus-zero rational curves of degree 11 that lie on projections of quartic scrolls of type $S_{2,2}$ determine an at-most 36-dimensional family. Consequently, such curves are nonproblematic.

We now treat rational curves of degree 11 on projections of quartic scrolls of type $S_{1,3}$. Every such curve pulls back to a rational, degree 11 curve on some quartic scroll $S_{1,3} \subset
\mathbb{P}^5$. Here $S_{1,3}$ is the image of $\mb{F}_2$, the Hirzebruch
surface $\mathbb{P}(\mc{O}_{\mb{P}^1} \oplus \mc{O}_{\mb{P}^1}(2))$ under the birational
map $\phi_{1,3}$
associated to the line bundle $\mc{O}_{\mb{F}_2}(e+3f)$, where $e$ (the class of
the unique section with negative self-intersection) and $f$ (the class
of a fiber)
satisfy $e^2=-2$, $f^2=0$ and $e \cdot
f=1$. 

A calculation analogous to the one carried
out above for $\mb{F}_0$ yields
$h^0(\mc{O}_{\mb{F}_2}(e+3f))=6$. Therefore, the projected scrolls
$\phi^{\prime}_{1,3}: \mb{F}_2 \rightarrow \mathbb{P}^4$, each given by 5 sections of $\mc{O}_{\mb{F}_2}(e+3f)$, vary in a
30-dimensional family. Quotienting by the
action of the 7-dimensional automorphism group of $\mb{F}_2$, we find that
projections of quartic scrolls of type $S_{1,3}$ comprise a
23-dimensional family. Since $e$ and $f$ are rational, the adjunction formula
implies the canonical class of $\mb{F}_2$ is
$K_{\mb{F}_2}= -2e-4f$. Similarly, letting $[C]=ae+bf$ denote the class of a rational curve of genus $g$
on $S_{1,3}$, adjunction implies 
\[
-2a(a-2) +a(b-4)+ b(a-2) =-2a^2+2ab-2b = 2g-2.
\]
Moreover, if $\mbox{deg }C=11$, then
\[
a+3b=11.
\]
As before, we may assume, furthermore, that $0 \leq g \leq 2$. It follows that
\[
12b^2-76b+120=-g,
\]
which forces $g=0$ and $(a,b)=(2,3)$. 

By the same argument we used to deduce
  the vanishing of higher cohomology of $\mc{O}_{\mb{F}_0}(e+mf)$ for all $m \geq 0$,
  we find that $h_0(\mc{O}_{\mb{F}_2}(2e+3f))=\chi(\mc{O}_{\mb{F}_2}(2e+3f))$. 
On the other hand, Riemann--Roch yields $\chi(\mc{O}_{\mb{F}_2}(2e+3f))=6$. It
follows that rational curves lying on projections of quartic scrolls
of type $(1,3)$ determine a 29-dimensional family and are, therefore, nonproblematic.

{\bf Veronese case.} Every Veronese quartic $V \subset \mathbb{P}^5$ is the image of $\mathbb{P}^2$ under the complete linear system 
$|\mc{O}_{\mathbb{P}^2}(2)|$. In particular, every divisor on $V$ has even degree,
so no rational curve of degree 11 arises via projection from a curve on $V$.

We now treat the remaining possibilities for $\deg(S)$. An important
upshot of our analysis will be that $h^0(\mathcal{I}_S(3)) \leq 4$ whenever $\mbox{deg }S \geq 5$.

\item {\bf $\mbox{deg }S=5$.} Let $\La$ denote a generic hyperplane section of $S$. Since $C$ lies on no hyperquadric, the standard exact sequence
\beg{equation}\label{stdidealseq}
0 \rightarrow \mc{I}_S(2) \rightarrow \mc{I}_S(3) \rightarrow \mc{I}_{\La}(3) \rightarrow 0 
\end{equation}
induces an inclusion $H^0(\mc{I}_S(3)) \hra H^0(I_{\La}(3))$. Whence, it suffices to bound the number of linearly independent cubics containing $\La$.

To this end, note that the curve $\La$ is 4-regular, by \cite[Thm. 1.1]{GLP}; accordingly, the standard exact sequence relating the ideal and structure sheaves of $\La$ induces a short exact sequence of spaces of global sections:
\beg{equation}\label{stdLa}
0 \ra H^0(\mc{I}_{\La}(3)) \ra H^0(\mc{O}_{\mb{P}^3}(3) \ra H^0(\mc{O}_{\La}(3)) \ra 0.
\end{equation}
Castelnuovo's genus bound \cite[p. 116]{ACGH} implies $g(\La) \leq 2$; accordingly, $K_{\La}(-3)$ has negative degree, and by Serre duality we conclude that
\[
h^1(\mc{O}_{\La}(3))= h^0(K_{\La}(-3))=0.
\]

It follows that $h^0(\mc{O}_{\La}(3))= 16-g(\La)$, and therefore, by \eqref{stdLa}, that
\[
h^0(\mc{I}_{\La}(3))= 4+ g(\La).
\]
In particular, if $g(\La)=0$, we conclude that 
\[
h^0(\mc{I}_S(3)) \leq h^0(\mc{I}_{\La}(3)) \leq 4.
\]

If $g(\La)=1$, we argue as follows. We may assume that the residual $S^{\pr}$ of $S$ is irreducible, since otherwise $S^{\pr}$ contains either a 2-plane or a quadric surface as an irreducible component, and we may argue as in case $\deg(S)=7$ or $\deg(S)=8$, respectively. By \cite[Cor. 5.2.14]{M}, it follows that $S^{\pr}$ is an irreducible surface of sectional genus $0$. In fact, we may further assume that $S^{\pr}$ is {\it reduced}. Otherwise, Bertini's theorem implies that $S^{\pr}$ is contained in the base locus $\mbox{Bs}(\mc{L})$, and whence, that 
\[
X \sub \mbox{Bs}(\mc{L})
\]
which violates Max Noether's $AF+BG$ theorem. On the other hand, any reduced irreducible surface of degree 4 in $\mb{P}^4$ is, by \cite[Thm. 2.1]{BrSc}, the projection of either a (possibly singular) quartic scroll or the Veronese surface in $\mb{P}^5$. Moreover, any such surface has a minimal free resolution of the form
{\small
\[
0 \ra \mc{O}_{\mb{P}^4}(-5) \ra \mc{O}_{\mb{P}^4}(-4)^{\oplus 4} \ra \mc{O}_{\mb{P}^4}(-2) \oplus \mc{O}_{\mb{P}^4}(-3)^{\oplus 3} \ra \mc{O}_{\mb{P}^4} \ra \mc{O}_{S^{\pr}} \ra 0
\]
}
in the quartic scroll case, and
{\small
\[
0 \ra \mc{O}_{\mb{P}^4}(-6) \ra \mc{O}_{\mb{P}^4}(-5)^{\oplus 5} \ra \mc{O}_{\mb{P}^4}(-4)^{\oplus 10} \ra \mc{O}_{\mb{P}^4}(-3)^{\oplus 7} \ra \mc{O}_{\mb{P}^4} \ra \mc{O}_{S^{\pr}} \ra 0
\]
}
in the Veronese case.
Using \cite[Prop. 5.2.10]{M}, we now deduce that the minimal free resolution of $S$ is either of the form
\[
\cdots \ra \mc{O}_{\mb{P}^4}(-2)^{\oplus 4} \oplus \mc{O}_{\mb{P}^4}(-3) \ra \mc{O}_{\mb{P}^4} \ra \mc{O}_{S} \ra 0,
\]
in which case $C$ is contained in a hyperquadric and, is therefore, nonproblematic, or
\[
\cdots \ra \mc{O}_{\mb{P}^4}(-1)^{\oplus 5} \oplus \mc{O}_{\mb{P}^4}(-3) \ra \mc{O}_{\mb{P}^4} \ra \mc{O}_{S} \ra 0,
\]
in which case $C$ is degenerate, which contradicts our initial assumptions.

Finally, if $g(\La)=2$, then $\La$ is ACM and contained in a quadric surface \cite[Thm. 3.7]{Ha}. It follows that $S$ is contained in a hyperquadric, in which case $C$ is nonproblematic.

\item {\bf $\mbox{deg }S=6$.} We may assume that the residual $S^{\prime}$
  is irreducible, since otherwise $S^{\prime}$ has a 2-plane as a
  component and we may argue as in the case $d=8$ below. Thus, $S^{\pr}$ is either a cubic scroll, or a cone over a twisted cubic. In either case, $S^{\pr}$ has Hilbert polynomial
\[
\chi_{S^{\pr}}(t)= 3 \binom{t+1}{2}+ \binom{t}{1}+ 1= \fr{3}{2} t^2+ \fr{5}{2}t+1.
\]
Using the basic
  liaison sequence \cite[p.551]{Na}
\[
0 \ra \omega_S(-1) \ra \mc{O}_X \ra \mc{O}_{S^{\prime}} \ra 0,
\]
Serre duality, and the fact that
\[
\chi_X(t)=\binom{t-2}{4}-2\binom{t+1}{4}+\binom{t+4}{4}=9/2t^2-9/2t+6,
\]
we deduce that
\[
\chi_S(t)=3t^2+t+1.
\]

Note that the holomorphic Euler characteristic of
  a generic hyperplane section of $S$ is given by
\[
\chi_{S \cap H} = \chi_S(0)-\chi_S(-1).
\]
Correspondingly, $S$ has sectional genus 
\[
1- \chi_{S \cap H}=3.
\]  
It follows from \cite[pp.175-6]{I} that either $\La$ is ACM, in which case $S$ is an ACM surface contained in 4
  linearly independent hypercubics, or the generic hyperplane section $\Lambda$
  of $S$ is a curve of type $(2,4)$ on a quadric surface $Q$. In the
  latter situation, $h^0(\mc{I}_C(3))$ is majorized by $h^0(\mc{I}_{\Lambda/\mathbb{P}^3}(3))$,
  which we compute as follows.

First, a chase through exact sequences shows that
\[
h^1(\mc{I}_{\Lambda/\mathbb{P}^3}(3))=h^1(\mc{I}_{\Lambda/Q}(3))=h^1(\mc{O}_Q(-1,1)),
\]
and the rightmost quantity is zero, by the K\"{u}nneth formula.
It follows that
\[
h^0(\mc{I}_{\Lambda/\mathbb{P}^3}(3))= h^0(\mc{O}_{\mathbb{P}^3}(3))-
h^0(\mc{O}_{\La}(3)).
\]
Here $h^0(\mc{O}_{\La}(3))=3 \cdot 6-3+1=16$ by Riemann--Roch. (We have applied Serre duality to deduce $h^1(\mc{O}_{\La}(3))=0$.) So $h^0(\mc{I}_{\Lambda/\mathbb{P}^3}(3))=4 \geq
h^0(\mc{I}_C(3))$.

\item {\bf $\mbox{deg }S=7$.} In this case, $S^{\prime}$ is a (possibly degenerate)
  quadric. We may assume that $S^{\prime}$ is, in fact, irreducible, since otherwise it contains a 2-plane as a component, and we may argue as in the case $\mbox{deg }S=8$. Since the structure sheaf of $S^{\pr}$ has a Koszul resolution
\[
0 \ra \mc{O}_{\mb{P}^4}(-3) \ra \mc{O}_{\mb{P}^4}(-1) \oplus \mc{O}_{\mb{P}^4}(-2) \ra \mc{O}_{\mb{P}^4} \ra \mc{O}_{S^{\pr}} \ra 0,
\]
we have
\[
\chi_{S^{\prime}}(t)=\binom{t+1}{4}- \binom{t+3}{4} -\binom{t+2}{4}+
\binom{t+4}{4}=t^2+2t+1.
\]
It follows that $\chi_S(t)=7/2t^2-1/2t+2$, and that the sectional genus
of $S$ is 5. Since $S$ (which contains $C$) is nondegenerate in
$\mathbb{P}^4$, its generic hyperplane section is an irreducible,
nondegenerate curve of degree 7 and genus 5 in $\mathbb{P}^3$. Since every such curve is ACM
~\cite[Prop. 1.2 and p.281]{PS} and contained in exactly three
hypercubics and no hyperquadrics, it follows that $S$ is contained in exactly three
hypercubics.

\item {\bf $\mbox{deg }S=8$.} In this case, $S^{\prime}$ is a
  2-plane. Via liaison, we deduce that
\[
\chi_{S}(t)=4t^2-3t+4.
\]
It follows, e.g., by the formula in ~\cite[Prop. 9]{Na}, that $S \cap
S^{\prime}$ is a plane quartic $F$. Note that the space of hypercubics containing
$F$ is four-dimensional (corresponding to those hypercubics containing
the plane spanned by $F$), so at most 4 linearly independent
hypercubics contain $S$.
\end{enumerate}

{\fl Now} assume that {\it $\deg(S_{i,j,k}) \geq 5$
  for all triples $(i,j,k)$}. If $(i_1,j_1,k_1)$ and $(i_2,j_2,k_2)$ are two
  distinct
  triples of indices that overlap in two places, say
  $(i_1,j_1)=(i_2,j_2)$, then the surface $S^*:=S_{i_1,j_1,k_1}
  \cup S_{i_2,j_2,k_2}$ lies in $X_{i_1} \cap X_{j_1}$. By degree
  considerations, it follows that
  $S_{i_1,j_1,k_1}=S_{i_2,j_2,k_2}$. So, by varying the choices of
  triples, we deduce that all of the $S_{i,j,k}$ agree,
  i.e., that the base locus $\mbox{Bs}(\mc{L})$ contains a
surface component $S$. This directly contradicts the fact (also
  demonstrated in the analysis above) that
  $h^0(\mathcal{I}_{S}(3))= h^0(\mathcal{I}_{S_{i,j,k}}(3)) \leq
  4$. Our proof is now complete.

\end{proof}

\subsection{Rational curves inside complete intersections of type $(3,3,3)$}\label{ci333}

In this section, we complete our study of $\mc{L}$ by proving the following result.

\beg{lem}\label{(3,3,3) curve case}
Let $C \sub \mb{P}^4$ denote a nondegenerate rational curve of degree 11, such that $g(C) \leq 2$, $C$ is contained in no hyperquadric, and $h^0(\mc{I}_C(3)) \geq 5$. Assume, moreover, that the base locus of the system $\mc{L}$ of linearly independent hypercubics containing $C$ contains no surface. Then $C$ is nonproblematic.
\end{lem}

\beg{proof}[Proof of Lemma~\ref{(3,3,3) curve case}.]
Let $X_1,X_2,X_3$ be three linearly
independent hypercubics containing $C$. Assume, for the sake of argument, that $X:= X_1 \cap X_2 \cap X_3$ is a
complete intersection curve. Since $\mc{O}_X$ has a free resolution of the form
\[
0 \rightarrow \mc{O}_{\mathbb{P}^4}(-9) \rightarrow
\mc{O}_{\mathbb{P}^4}(-6)^{\oplus 2} \rightarrow
\mc{O}_{\mathbb{P}^4}(-3)^{\oplus 3} \rightarrow \mc{O}_{\mathbb{P}^4}
\rightarrow \mc{O}_X \rightarrow 0,
\]
additivity of the Hilbert polynomial yields
$\chi_X(t)=27t-54$. Let $C^{\prime}$ denote the residual of $C$ in
$X$; then by the Hartshorne--Schenzel theorem ~\cite[Theorem 5.3.1]{M},
$H^1(\mc{I}_C)_{*}$ and 
\begin{equation} \mbox{Hom}_{\mb{C}}(H^1(\mc{I}_{C^{\pr}}),\mb{C})(-4)_{*}
\notag \end{equation} are
duals, where the $k$th graded piece of the module on the right is
\[
\mbox{Hom}_{\mb{C}}(H^1(\mc{I}_{C^{\pr}}(4-k)),\mb{C}). 
\]
Whence,
$h^1(\mc{I}_C(t))=h^1(\mc{I}_{C^{\prime}}(4-t))$ for all integers $t$. In particular, we
have 
\[
h^1(\mc{I}_{C^{\prime}}(-1)) \neq 0.
\]
On the other hand, the long exact sequence in cohomology associated to the standard sequence that relates the structure and ideal sheaves of $C$ yields
\[
h^1(\mc{I}_{C^{\prime}}(-1))=h^0(\mc{O}_{C^{\prime}}(-1)).
\]
It follows that $C^{\prime}$
contains a maximal nonreduced subscheme $Y$, whose degree is at least
2. By Bertini's theorem, $\mc{L}$ is smooth away from its base locus
$\mbox{Bs}(\mc{L})$, so we
conclude that $C \cup Y \subset \mbox{Bs}(\mc{L})$. 

{\fl Now} let $C^{\prime}=Y \cup Y^{\pr}$, where $Y$ and
$Y^{\pr}$ have no common irreducible component.

{\fl Note} that $Y^{\pr}$ is nontrivial, since otherwise
$\mbox{Bs}(\mathcal{L})$ contains all of $X$. This is absurd, since
$\mbox{Bs}(\mathcal{L})$ is the intersection of at least five hypercubics and $X$
is the complete intersection of three of these.

{\fl Moreover}, we have
\begin{equation*}
h^1(\mc{I}_{C \cup Y}(5))=h^1(\mc{I}_{Y^{\pr}}(-1))=0
\end{equation*}
by the Hartshorne--Schenzel theorem, together with the fact that
$Y^{\pr}$ is reduced. Likewise,
\begin{equation*}
h^2(\mc{I}_{C \cup Y}(4))=h^0(\mc{I}_{Y^{\pr}})-h^0(\mc{I}_X)=0,
\end{equation*}
by \cite[Cor. 3, p.552]{Na}.

{\fl It} follows that $C \cup Y$ is 6-regular; in particular, $\mbox{gin}(\mc{I}_{C \cup Y})$ is generated in degrees at most six. On the other hand, we have an
inclusion
\beg{equation}\label{gininj}
\mbox{gin}(\mc{I}_{C \cup Y}) \subset \mbox{gin}(\mc{I}_C)
\end{equation}
of generic initial ideals. The inclusion \eqref{gininj} imposes restrictions on $\mbox{gin}(\mc{I}_C)$ which will enable us to conclude that $C$ is nonproblematic.

{\fl To} this end, let $H \sub \mb{P}^4$ denote a general hyperplane, and let $\wt{\Ga}= (C \cup Y) \cap H$ and $\Ga^{\pr}=Y^{\pr} \cap H$ denote the corresponding hyperplane sections of $C \cup Y$ and $Y^{\pr}$, respectively, that are residual to one another inside $Z:=X \cap H$. The Cayley--Bacharach theorem \cite[Thm CB7, p.313]{EGH} implies that
\beg{equation}\label{hs1}
h^0(\mc{I}_{\Ga^{\pr}/H}(t))- h^0(\mc{I}_{Z/H}(t))= h^1(\mc{I}_{\wt{\Ga}/H}(5-t))
\end{equation}
for all integers $t$. 

\beg{claim}\label{h1la4claim}
We have
\beg{equation}\label{h1la4}
h^1(\mc{I}_{\Ga/H}(4)) \leq 1.
\end{equation}
\end{claim}

\beg{proof}[Proof of Claim~\ref{h1la4claim}.]
The scheme $Z$ is linearly nondegenerate. Consequently, \eqref{hs1} implies that
\[
h^1(\mc{I}_{\wt{\Ga}/H}(4))= h^0(\mc{I}_{\Ga^{\pr}/H}(1)).
\]
Now assume \eqref{h1la4claim} fails to hold. Then
\[
h^0(\mc{I}_{\Ga^{\pr}/H}(1)) \geq 2,
\]
i.e., $\Ga^{\pr}$ is supported along a line. Since $\Ga^{\pr}$ is a generic hyperplane section of $Y^{\pr}$, which is reduced, it follows that $\Ga^{\pr}$ is necessarily a point. In particular, since $h^1(\mc{I}_{\Ga^{\pr}/H}(2))$ computes the number of $\La$-rewritings applied in degree 3 or greater in forming $T(\mbox{gin}(\mc{I}_{\Ga^{\pr}/H}))$, by Fact~\ref{1.10} of Section~\ref{curvegins}, we have
\[
h^1(\mc{I}_{\Ga^{\pr}/H}(2))=0.
\]

{\fl On} the other hand, \eqref{hs1} implies that
\[
\beg{split}
h^1(\mc{I}_{\Ga^{\pr}/H}(2))= h^0(\mc{I}_{\wt{\Ga}/H}(3))- h^0(\mc{I}_{Z/H}(3)) \geq 2,
\end{split}
\]
which is a contradiction.
\end{proof}


{\fl A} key inference is now as follows. As usual, let $\Ga= C \cap H$.
\beg{claim}\label{key inference}
In forming the tree of minimal generators of $\mbox{gin}(\mc{I}_{C})$ from that of $\mbox{gin}(\mc{I}_{\Ga})$, no $C$-rewriting rule in degree 6 may be applied to a minimal generator of the form $x_0^ix_1^jx_2^k$, where $i+j+k=6$. Further, the following sequence of rewritings occurs at most once: $x_0^ix_1^jx_2^k$ of degree 5 is rewritten, and $x_0^ix_1^jx_2^kx_3$ of degree 6 is rewritten. 
\end{claim}

\beg{proof}[Proof of Claim~\ref{key inference}.]
To verify the first assertion, assume, for the sake of argument, that a rewriting rule in degree 6 were applied to a vertex corresponding to some minimal generator $x_0^ix_1^jx_2^k$, where $i+j+k=6$. By Borel-fixity, $\mbox{gin}(\mc{I}_{C \cup Y})$, which is contained in $\mbox{gin}(\mc{I}_C)$, would necessarily contain a minimal generator divisible by $x_0^ix_1^jx_2^{k+1}$, thereby contradicting the 6-regularity of $\mc{I}_{C \cup Y}$.

Similarly, to verify the second assertion, assume that two rewriting rules (in degrees 5 and 6, as described) were used. By Borel-fixity, $\mbox{gin}(\mc{I}_{C \cup Y})$ would then necessarily contain two minimal generators of the form $x_0^ix_1^jx_2^k$ in degree six. Since $\mbox{gin}(\mc{I}_{C \cup Y})$ contains no minimal generators in degrees $\geq 7$, those minimal generators would belong to the saturation of $\mbox{gin}(\mc{I}_{C \cup Y})$ with respect to $x_3$, i.e., to $\mbox{gin}(\mc{I}_{\Ga})$, which violates \eqref{h1la4}.
\end{proof}

{\fl Our} next task is to establish an upper bound on the dimension of the space $\mc{L}$ of hypercubics containing a nondegenerate rational curve $C$, assuming as usual that $C$ lies on a general quintic. In particular, this means that $C$ lies on no hyperquadric.
\beg{claim}\label{h07}
Either
\beg{equation}\label{lastest}
h^0(\mc{I}_C(3)) \leq 7,
\end{equation}
or $C$ is nonproblematic.
\end{claim}

\beg{proof}[Proof of Claim~\ref{h07}.]
Since $C$ lies on no hyperquadric, by assumption, the long exact sequence in cohomology obtained from the natural exact sequence of sheaves
\[
0 \ra \mc{I}_{C \cup Y}(2) \ra \mc{I}_{C \cup Y}(3) \ra \mc{I}_{\wt{\Ga}/H} \ra 0
\]
yields $h^0(\mc{I}_{C \cup Y}(3)) \leq h^0(\mc{I}_{\wt{\Ga}/H}(3))$. On the other hand, we have 
\beg{equation}\label{basicid}
h^0(\mc{I}_C(3))= h^0(\mc{I}_{C \cup Y}(3))
\end{equation}
since $C \cup Y \sub \mbox{Bs}(\mc{L})$. Whence,
\beg{equation}\label{basicineq}
h^0(\mc{I}_C(3)) \leq h^0(\mc{I}_{\wt{\Ga}/H}(3)).
\end{equation}
Now assume that $h^0(\mc{I}_C(3)) \geq 8$. Then \eqref{basicineq} yields 
\beg{equation}\label{h0Ga}
h^0(\mc{I}_{\wt{\Ga}/H}(3)) \geq 8.
\end{equation}
The latter inequality, together with \eqref{hs1}, forces the nonreduced locus $Y$ to be rather small. More precisely, letting $m:=\deg(Y)$:
\beg{claim}\label{mleq4}
We have $m \leq 4$. 
\end{claim}

\beg{proof}[Proof of Claim~\ref{mleq4}.] The argument naturally divides into two cases, depending on the number of linearly independent quadrics containing $\wt{\Ga}$.
{\fl \bf Case: $h^0(\mc{I}_{\wt{\Ga}/H}(2))=0$.} From \eqref{hs1}, we deduce that
\beg{equation}\label{CBid1}
h^1(\mc{I}_{\wt{\Ga}/H}(4))=h^0(\mc{I}_{\Ga^{\pr}}(1)).
\end{equation}
Since $\Ga^{\pr}$, the generic hyperplane section of $Y^{\pr}$, is reduced, we have $h^0(\mc{I}_{\Ga^{\pr}}(1)) \leq 1$ unless $\Ga^{\pr}$ is a point. In fact, the latter possibility is impossible, as it would imply that
\[
0=h^1(\mc{I}_{\Ga^{\pr}}(1))= h^0(\mc{I}_{\wt{\Ga}/H}(4)),
\]
which is clearly absurd.
Whence, \eqref{CBid1} implies that
\beg{equation}\label{h1Ga}
h^1(\mc{I}_{\wt{\Ga}/H}(4)) \leq 1.
\end{equation}
Now recall that $\deg(\wt{\Ga})$ (resp., $h^1(\mc{I}_{\wt{\Ga}/H}(4))$) computes the number of $\La$-rewritings (resp., the number of $\La$-rewritings in degrees at least 5) applied in forming the tree $T(\mbox{gin}(\mc{I}_{\wt{\Ga}/H}))$ of minimal generators. Borel-fixity, $h^0(\mc{I}_{\wt{\Ga}/H}(2))=0$, and \eqref{h1Ga} now force $\deg(\wt{\Ga}) \leq 15$, with equality if and only if
\[
\mbox{gin}(\mc{I}_{\wt{\Ga}/H})= \mbox{Borel}(x_2^6,x_1x_2^4,x_1^2x_2,x_0x_2^2).
\]
The claim follows immediately in this case.
{\bf \fl Case: $h^0(\mc{I}_{\wt{\Ga}/H}(2)) \neq 0$.} For simplicity, begin by considering the case $h^0(\mc{I}_{\wt{\Ga}/H}(2))=1$. By Borel-fixity, $T(\mbox{gin}(\mc{I}_{\wt{\Ga}/H}))$ is obtained via $m+1$ $\La$-rewritings from the tree of minimal generators of the degree-10 ideal $\mbox{Borel}(x_2^3,x_0^2)$. The fact that $h^0(\mc{I}_{\wt{\Ga}/H}(3)) \geq 8$ implies that at least $m-2$ of these rewritings occur in degrees at least 4, i.e., that
\beg{equation}\label{m-2}
h^1(\mc{I}_{\wt{\Ga}/H}(3)) \geq m-2.
\end{equation}
Applying \eqref{hs1}, we deduce that
$h^0(\mc{I}_{\Ga^{\pr}/H}(2)) \geq m-2$. In particular, if $h^0(\mc{I}_{\Ga^{\pr}/H}(2)) \leq 2$, then $m \leq 4$. The remaining possibility is that $h^0(\mc{I}_{\Ga^{\pr}/H}(2))=3$, in which case any three linearly independent quadrics containing $\Ga^{\pr}$ define a twisted cubic or degeneration thereof. By Borel-fixity, it follows that $x_0^2$, $x_0x_1$, and $x_1^2$ are minimal generators of $\mbox{gin}(\mc{I}_{\Ga^{\pr}/H})$. On the other hand, \eqref{hs1}, together with the fact that $\wt{\Ga}$ is linearly nondegenerate, implies that
\beg{equation}\label{Gaprime4}
h^1(\mc{I}_{\Ga^{\pr}/H}(4))=0.
\end{equation}
As a result of \eqref{Gaprime4}, $\mbox{gin}(\mc{I}_{\Ga^{\pr}/H})$ has no minimal generators of degree 6 or higher. On the other hand, \eqref{hs1}, together with the fact that $h^0(\mc{I}_{\wt{\Ga}/H}(3)) \geq 8$, implies that
\beg{equation}\label{Gaprime2}
h^1(\mc{I}_{\Ga^{\pr}/H}(2)) \geq 5,
\end{equation}
i.e., at least five $\La$-rewritings in degrees at least 5 are applied in forming $T(\mbox{gin}(\mc{I}_{\Ga^{\pr}/H}))$. It now follows that $\mbox{gin}(\mc{I}_{\Ga^{\pr}/H})$ is either
\[
\mbox{Borel}(x_2^5,x_1x_2^4,x_1^2,x_0x_2^3,x_0x_1), \text{ or } \mbox{Borel}(x_2^5,x_1x_2^4,x_1^2,x_0x_2^4,x_0x_1).
\]
These two ideals are of degrees 14 and 15, respectively. In fact, the second possibility is excluded, since in that case we have $\deg(\wt{\Ga})=12$, which is too small. In any case, we deduce that $m \leq 4$ whenever $h^0(\mc{I}_{\wt{\Ga}/H}(2))=1$.

{\fl Finally}, assume $h^0(\mc{I}_{\wt{\Ga}/H}(2)) =j \geq 2$. Borel-fixity implies that $T(\mbox{gin}(\mc{I}_{\wt{\Ga}/H}))$ is obtained via at least $m+j$ $\La$-rewritings from the tree of a 3-regular ideal whose degree-2 graded piece agrees with that of $\mbox{gin}(\mc{I}_{\wt{\Ga}/H})$. The fact that $h^0(\mc{I}_{\wt{\Ga}/H}(3)) \geq 8$ implies that at least $m-2$ of these rewritings occur in degrees at least 4. We now conclude, exactly as before, that $m \leq 4$.
\end{proof}

{\fl From} Claim~\ref{mleq4}, it follows that the nonreduced scheme $Y$ is supported along either a line or a conic. If $Y$ is supported along a conic, then it is necessarily the square of a plane conic, which has genus 3. The remaining possibilities for $Y$ are stratified according to embedding dimension. If $\mbox{emb. dim.}(Y)=4$, the ideal of the generic hyperplane section $\Ga^*= Y \cap H$ has no linear minimal generators; consequently,
\[
\mbox{gin}(\mc{I}_{\Ga^*/H})= \mbox{Borel}(x_2^2).
\]
It follows that $g(Y) \leq 0$ in this case. 

{\fl The} remaining possibilities for $Y$ are as follows \cite{Hart}. If $\mbox{emb. dim.}(Y)=3$, then $g \leq \fr{(m-2)(m-3)}{2}$; in particular, we have $g \leq 0$ unless $m=4$ and $g=1$. On the other hand, if $\mbox{emb. dim.}(Y)=2$, then $g(Y) \leq 3$; when $m \leq 3$, we have $g(Y) \leq 1$.

{\fl Since} the basis of our argument is the basic identity \eqref{basicid}, which uses only the fact that $C$ and $Y$ are components of $\mbox{Bs}(\mc{L})$, we may assume that no components of $Y^{\pr}$ belong to $\mbox{Bs}(\mc{L})$. By Bezout's theorem, we have
\beg{equation}\label{48-3m}
\deg(C \cap Y^{\pr}) \leq 48-3m;
\end{equation}
otherwise, every hypercubic containing $C$ contains a component of $Y^{\pr}$.

{\fl Comparing} $\wt{C}$ with its residual $Y^{\pr}$ using \cite{Na}, we now find
\beg{equation}\label{liaison}
2g(\wt{C})-2= 44+4m- \deg(\wt{C} \cap Y^{\pr}).
\end{equation}
Combining \eqref{liaison} with \eqref{48-3m}, we deduce
\beg{equation}\label{g(wtC)}
g(\wt{C}) \geq \fr{7}{2}m-1.
\end{equation}

{\fl On} the other hand, the holomorphic Euler characteristics $\chi$ of $C$, $Y$, and $\wt{C}= C \cup Y$ are related by
\beg{equation}\label{eulerchar}
\beg{split}
\chi(C \cap Y)&= \chi(C)+ \chi(Y)- \chi(\wt{C}), \text{ i.e.,} \\
\deg(C \cap Y)&= g(\wt{C})- g(C)- g(Y) +1.
\end{split}
\end{equation}
Together, \eqref{g(wtC)} and \eqref{eulerchar} imply
\beg{equation}\label{degest}
\deg(C \cap Y) \geq \fr{7}{2}m-g-g(Y).
\end{equation}

{\fl We} now apply \eqref{degest} to estimate the codimension of the subspace of $M_{11}$ determined by the corresponding rational curves $C$. Here we may assume $g \leq 2$, since otherwise $C$ is nonproblematic, by the analysis of Section~\ref{beginofproof}. To obtain our codimension estimate, we use the following fact, which refines the standard description of the tangent space to $M_d$ at an unobstructed point $f$.

{\bf \fl Key fact (\cite[p. $45$]{D}).} Let $B \sub \mb{P}^n$ denote a curve, and let $D_1 \in \mbox{Sym}^{\ga} \mb{P}^1$ and
$D_2 \in \mbox{Sym}^{\ga} B$ denote divisors of equal degree $\gamma$. The subspace of $M_d^n$ determined by morphisms
$f: \mb{P}^1 \rt B$ with $f^*(D_2)=D_1$ has dimension
$h^0(f^*T_{\mb{P}^n} \otimes \mc{I}_{D_1})$ whenever
$h^1(f^*T_{\mb{P}^n} \otimes \mc{I}_{D_1})=0$. 

{\fl In} our case, we may further assume that every degree-11 morphism $f: \mb{P}^1 \ra \mb{P}^4$ under consideration has RTB splitting type $(4,3,2,2)$, corresponding to the codimension-2 RTB stratum. Indeed, elements of the generic RTB stratum $(3,3,3,2)$ are 6-regular, while more special strata are of codimension $\geq 4$; our uniform estimate $i \leq 3$ implies that the corresponding rational curves are nonproblematic. Consequently, we have
\beg{equation}\label{ressplitting}
f^*T_{\mb{P}^4} \cong \mc{O}_{\mb{P}^1}(15) \oplus \mc{O}_{\mb{P}^1}(14) \oplus \oplus \mc{O}_{\mb{P}^1}(13)^{\oplus 2}.
\end{equation}

\beg{itemize}
\item{\bf Case: $m=2$.} Since $g \leq 2$ and $g(Y) \leq 0$, the degree estimate \eqref{degest} implies
\[
\deg(C \cap Y) \geq 5.
\]
For the sake of codimension estimation, we may assume the inequality is in fact an equality. Consequently, $C \cap Y= \sum_i m_i p_i$ is supported in a collection of points $p_i$ along $L= Y_{\mbox{red}}$, with total multiplicity $\sum_i m_i=5$. Now fix a choice of line $L \sub \mb{P}^4$, points $\{p_i\}$, and targets $f(p_i) \in L$. The key fact cited above, together with \eqref{ressplitting}, implies that we obtain $4 \cdot 5=20$ linearly independent conditions. Since lines vary in a 6-dimensional family, while $\{p_i\}$ and $\{f(p_i)\}$ each vary in at most 5-dimensional families (corresponding to symmetric squares of $\mb{P}^1$ and $L$, respectively), we deduce that the corresponding rational curves determine a subspace of $M_{11}$ of codimension at least 4. Note that this argument recovers the result of Lemma~\ref{1a}.

\item {\bf Case: $m=3$.} Here $g \leq 2$ and $g(Y) \leq 1$. Correspondingly, \eqref{degest} implies
\[
\deg(C \cap Y) \geq 8.
\]
As in the preceding subcase, it suffices to treat the case $\deg(C \cap Y) = 8$. The corresponding rational curves determine a subspace of $M_{11}$ of codimension at least $2(8)-6=10$.

\item {\bf Case: $m=4$.} Here $g \leq 2$ and $g(Y) \leq 3$, so \eqref{degest} implies
\[
\deg(C \cap Y) \geq 9.
\]
Without loss of generality, we may assume $\deg(C \cap Y) = 9$. Whence, $C \cap Y= \sum_i m_i p_i$ is supported in a collection of points $p_i$ along $Y_{\mbox{red}}$, with total multiplicity $\sum_i m_i=9$. Here $Y_{\mbox{red}}$ is either a line, a conic, or a pair of intersecting lines. In the first case, the same enumeration used in the first two subcases shows that the corresponding $C$ determine a subspace of $M_{11}$ of codimension at least 12. Similarly, in the second case, fixing a choice of conic in $\mb{P}^4$, points $p_i$, and their targets $f(p_i)$, we obtain 36 linearly independent conditions. Conics in $\mb{P}^4$ vary in a 9-dimensional family; accordingly, the corresponding subspace of $M_{11}$ is of codimension at least 9. Finally, in the third case, we use the fact that pairs of intersecting lines in $\mb{P}^4$ determine an 11-dimensional family to deduce that the corresponding subspace of $M_{11}$ is of codimension at least 7.
\end{itemize}
\end{proof}

{\fl Our} next task is to show that we can improve our upper estimate on $i$ by applying Claim~\ref{key inference}, in tandem with \eqref{lastest}.

\beg{claim}\label{improved i-estimate}
Let $C \sub \mb{P}^4$ denote a nondegenerate rational curve of degree 11, for which $g(C) \leq 2$, $\mc{I}_C$ is 7-regular, $C$ lies on no hyperquadric, and $5 \leq h^0(\mc{I}_C(3)) \leq 7$. Then $i \leq 2$.
\end{claim}
\beg{proof}[Proof of Claim~\ref{improved i-estimate}.]
As usual, we argue case-by-case, based on the possibilities for $\mbox{gin}(\mc{I}_{\Ga})$.
\beg{itemize}
\item {\bf Case 1:} $\mbox{gin}(\mc{I}_{\Gamma})=\mbox{Borel}(x_2^4,x_1x_2^2)$. Here $g+i \leq 8$. Consider possible curve gins $I$ obtained by applying 6 $C$-rewritings to $\mbox{gin}(\mc{I}_{\Gamma})=\mbox{Borel}(x_2^4,x_1x_2^2)$. Since at most 3 of these may occur in degrees at least 7, any such ideal has $i \leq 3$, with equality if and only if
\[
I= \mbox{Borel}(x_2^4x_3^3,x_1x_2^2).
\]
However, the latter ideal is in fact disallowed, by the key inference, since it is obtained by applying a rewriting rule in degree 6 to a polynomial in $x_0$, $x_1$, and $x_2$ alone. An easy inspection shows that admissible curve gins all satisfy $i \leq 1$.

\item {\bf Case 2:} Either
\beg{itemize}
\item $\mbox{gin}(\mc{I}_{\Ga})=\mbox{Borel}(x_2^4,x_1^2x_2,x_0^2)$,
or
\item $\mbox{gin}(\mc{I}_{\Ga})= \mbox{Borel}(x_2^5,x_1x_2^2,x_0^2)$.
\end{itemize}
To maximize $i$, we assume $g=0$. In the first subcase, $g(C_{\Ga})=9$. We begin by applying a $C$-rewriting at the vertex corresponding to the unique quadratic generator of $\mbox{gin}(\mc{I}_{\Ga})$, since curves that lie on hyperquadrics may be ignored. Likewise, since the resulting ideal contains 9 cubic minimal generators, we apply 2 more $C$-rewriting rules in degree 3. After applying 5 more $C$-rewriting rules (for a total of 8), we have rewritten at most twice in degree 6. For example, a typical ideal arising at this stage is
\[
I= \mbox{Borel}(x_0^2x_3,x_0x_2^2,x_1^2x_2x_3,x_1x_2^3,x_2^6,x_2^4x_3^3),
\]
which defines a scheme with $g=1$ and $i=2$.
We only have 1 more $C$-rewriting at our disposal, which is necessarily in degree less than 6, by the key inference. So $i \leq 2$.

\medskip
In the second subcase, $g(C_{\Ga})=10$. As in the first subcase, we begin by applying a $C$-rewriting at the vertex corresponding to the unique quadratic generator of $\mbox{gin}(\mc{I}_C)$. Likewise, since the resulting ideal contains 10 cubic minimal generators, we apply 3 more $C$-rewriting rules in degree 3. After applying 2 more $C$-rewriting rules (for a total of 6), we have rewritten at most once in degree 6. (This is possible because $\mbox{gin}(\mc{I}_{\Gamma})$ has a minimal generator in degree 5.) Applying the key inference, we find that after applying 3 more $C$-rewriting rules, we have rewritten at most twice in degree 6. We only have 1 more $C$-rewriting at our disposal, which is necessarily in degree less than 6, by the key inference. So $i \leq 2$ once more.

\item {\bf Case 3: } Either
\beg{itemize} 
\item $\mbox{gin}(\mc{I}_{\Ga})=\mbox{Borel
}(x_0x_1,x_0x_2^3,x_1^2x_2,x_2^4)$,
\item $\mbox{gin}(\mc{I}_{\Ga})=\mbox{Borel
}(x_0x_1,x_0x_2^2,x_1^3,x_2^4)$,
\item $\mbox{gin}(\mc{I}_{\Ga})=\mbox{Borel
}(x_0x_1,x_0x_2^2,x_1^2x_2,x_1x_2^3,x_2^5)$,
or
\item $\mbox{gin}(\mc{I}_{\Ga})=\mbox{Borel
}(x_0x_1,x_0x_2^3,x_1x_2^2,x_2^5)$.
\end{itemize}
In the first two (resp., last two) subcases, we have $g(C_{\Ga})=10$ (resp., $g(C_{\Ga})=11$).
Assume $g=0$. In each subcase, we begin by applying 2 $C$-rewritings at those vertices corresponding to the quadratic minimal generators of $\mbox{gin}(\mc{I}_{\Ga})$.

In the first two subcases, the resulting ideal has 9 cubic minimal generators. We apply $C$-rewritings to two of those in order to satisfy \eqref{lastest}. We now have 6 $C$-rewritings at our disposal; after 5 of these, we have rewritten at most twice in degree 6. A typical example of an ideal arising at this stage is
\[
I= \mbox{Borel}(x_0x_1x_3,x_0x_2^3,x_1^2x_2x_3,x_1x_2^3x_3,x_2^4x_3^3,x_2^6),
\]
for which $g=1$ and $i=2$. We only have 1 more $C$-rewriting at our disposal, which necessarily occurs in degree less than 6, by the key inference. So $i \leq 2$ in the first two subcases.

Similarly, in the third and fourth subcases, after applying rewritings at the quadratic generators, the resulting ideal has 10 minimal cubic generators. We apply $C$-rewritings to three of those in order to satisfy \eqref{lastest}. We now have 6 $C$-rewritings at our disposal, of which 2 may occur in degree 6, by the key inference. We deduce that $i \leq 2$. A typical example of an ideal that might arise is
\[
I= \mbox{Borel}(x_0x_1x_3,x_0x_2^2x_3,x_1^2x_2x_3,x_1x_2^4x_3,x_1x_2^3x_3^3,x_2^6,x_2^5x_3^2).
\]

\item {\bf Case 4: } 
Either
\beg{itemize}
\item $\mbox{gin}(\mc{I}_{\Ga})= \mbox{Borel}(x_0x_2,x_2^4)$, in which case $g(C_{\Ga})=11$; 
\item $\mbox{gin}(\mc{I}_{\Ga})= \mbox{Borel}(x_2^5,x_1x_2^3,x_1^3,x_0x_2)$, in which case $g(C_{\Ga})=13$; 
\item  $\mbox{gin}(\mc{I}_{\Ga})=\mbox{Borel}(x_2^5,x_1^2x_2,x_0x_2)$, in which case $g(C_{\Ga})=12$.
\end{itemize}
Assume $g=0$. In each subcase, we begin by applying 3 $C$-rewritings at those vertices corresponding to the quadratic minimal generators of $\mbox{gin}(\mc{I}_{\Ga})$.

In the first subcase, the resulting ideal has 9 minimal cubic generators. We apply $C$-rewritings to two of those in order to satisfy \eqref{lastest}. We now have 6 $C$-rewritings at our disposal, of which 2 may occur in degree 6, by the key inference. Whence, $i \leq 2$. A typical example of an ideal that might arise is
\[
I= \mbox{Borel}(x_0x_1x_3,x_0x_2x_3^2,x_1^2x_2^2,x_1x_2^3x_3,x_2^6,x_2^5x_3^2,x_2^4,x_3^3).
\]

In the second subcase, the ideal that is the result of applying 3 $C$-rewritings to the quadratic minimal generators of $\mbox{gin}(\mc{I}_{\Ga})$ contains 10 minimal cubic generators. We apply $C$-rewritings to three of those in order to satisfy \eqref{lastest}, leaving us with 7 $C$-rewritings at our disposal. Of those, at most two may occur in degree 6, by the key inference. Whence, $i \leq 2$. A typical example of an ideal that might arise is
\[
I= \mbox{Borel}(x_0x_1x_3,x_0x_2x_3^2,x_1^3x_3,x_1^2x_2^3,x_1^2x_2^2x_3^2,x_1x_2^4,x_1x_2^3x_3^3,x_2^6,x_2^5x_3^2).
\]

Finally, in the third subcase, the ideal that is the result of applying 3 $C$-rewritings to the quadratic minimal generators of $\mbox{gin}(\mc{I}_{\Ga})$ contains 11 minimal cubic generators. We apply $C$-rewritings to four of those in order to satisfy \eqref{lastest}, leaving us with 5 $C$-rewritings at our disposal. Of those, at most two may occur in degree 6, by the key inference. Whence, $i \leq 2$. A typical example of an ideal that might arise is
\[
I= \mbox{Borel}(x_0x_1x_3,x_0x_2x_3^2,x_1^3x_3,x_1^2x_2^2,x_1^2x_2x_3^4,x_1x_2^4,x_2^6,x_2^5x_3^2).
\]
\end{itemize}
\end{proof}
Because $i \leq 2$ in every case, it follows that $C$ is nonproblematic whenever $g=0$. On the other hand, if $i>0$, then the restricted tangent bundle $f^* \mc{T}_{\mb{P}^4}$ of the map $f: \mb{P}^1 \ra \mb{P}^4$ with image $C$ necessarily splits nongenerically. The proof of Theorem~\ref{codlemma}, sketched in Section~\ref{sec-five} and given in detail in the earlier preprint \cite{Cot2}, shows that {\it those maps whose images have arithmetic genus $g$ determine a subspace of codimension at least $\min(2g,12)$ within each RTB stratum}. In particular, if $g>0$ and $i>0$, then $C$ belongs to a subspace of codimension at least $(2+ 2g)$, which is greater than $(g+i)$ whenever $i \leq 2$. So once more, we deduce that $C$ is nonproblematic.
\end{proof}

\subsection{Castelnuovo curves are nonproblematic}
In order to complete the proof of Theorem~\ref{1.1bis}, it remains, by Theorem~\ref{singest}, to show that nondegenerate {\it Castelnuovo curves}, i.e., those singular rational curves $C$ of maximum arithmetic genus $12$, are nonproblematic. Any such curve lies on a cubic scroll in $\mathbb{P}^4$. Cubic scrolls are of two basic types; those that are smooth will be denoted $S_{1,2}$, while those that are cones over twisted cubic curves will be denoted $S_{0,3}$. 

{\fl We} begin by considering those rational curves of degree $11$ lying on smooth cubic scrolls  $S_{1,2}$. Every $S_{1,2}$ may be realized as the image of the Hirzebruch surface $\mathbb{F}_1:= \mathbb{P}(\mathcal{O}_{\mathbb{P}^1} \oplus  \mathcal{O}_{\mathbb{P}^1}(1))$ under the map $\phi_{1,2}$ defined by the complete linear series $|\mathcal{O}_{\mathbb{F}_1}(e+2f)|$, where $e$ is the divisor class of the $-1$-curve on $\mathbb{F}_1$ and $f$ is the class of the fiber.

{\fl The} intersection pairing on $\mathbb{F}_1$ is given by
\begin{equation*}
e^2=-1, e \cdot f=1, f^2=0.
\end{equation*}
Moreover, the
canonical class of $\mathbb{F}_1$ is
\begin{equation*}
K_{\mathbb{F}_1}=-2e-3f.
\end{equation*}

{\fl Now} let $\phi_{1,2}^{\star}[C]=ae+bf \in \mbox{Pic }\mathbb{F}_1$. The adjunction formula implies
that the arithmetic genus $g$ of $C$ satisfies
\begin{equation}\label{adjunction_on_S_{1,2}}
2g-2=((a-2)e+(b-3)f)\cdot(ae+bf),
\end{equation}
i.e.,
\begin{equation}\label{S_{1,2}eq1}
a^2-2ab+a+2b+2g-2=0.
\end{equation}
Also, since $\deg(C)=11$, we have
\begin{equation*}
(e+2f) \cdot (ae+bf)=11,
\end{equation*}
or equivalently,
\begin{equation}\label{S_{1,2}eq2}
b=11-a.
\end{equation}
Substituting \eqref{S_{1,2}eq2} in \eqref{S_{1,2}eq1}, we obtain
\begin{equation}\label{S_{1,2}eq3}
3a^2-23a+20+2g=0.
\end{equation}
Substituting $g=12$ in \eqref{S_{1,2}eq3} we obtain $a=4$. Then \eqref{S_{1,2}eq2} implies $b=7$. To estimate the dimension of curves of class $4e+7f$ on $\mathbb{F}_1$, we proceed as follows.

{\fl First,} note that for any $n \geq 0$, the dimension of any component of the space of maps $\mb{P}^1 \ra \mb{F}_1$ with image of class $(ae+bf)$ that contains a given map $\psi$ is at most $h^0(\mc{N}_{\psi}/(\mc{N}_{\psi})_{\mbox{tors}}))$, where 
$\mc{N}_{\psi}$ is the normal sheaf of the map $\psi$ and $(\mc{N}_{\psi})_{\mbox{tors}}$ is its torsion subsheaf \cite[Lem. 3.41]{HM}. On the other hand, the short exact sequence
\[
0 \ra T_{\mb{P}^1} \ra \psi^* T_{\mb{F}^1} \ra \mc{N}_{\psi} \ra 0
\]
that defines the normal sheaf implies that
\[
\beg{split}
\deg(\mc{N}_{\psi}/(\mc{N}_{\psi})_{\mbox{tors}})) &\leq \deg(-\psi^* K_{\mb{F}_1}) -2 \\
& \leq (2e+3f) \cdot (ae+bf) - 2 \\
&= a+2b -2.
\end{split}
\]
In particular, when $(a,b)=(4,7)$, we have
\[
h^0(\mc{N}_{\psi}/(\mc{N}_{\psi})_{\mbox{tors}}))=17.
\]

{\fl On} the other hand, the dimension of the space of cubic scrolls $S_{1,2} \subset \mathbb{P}^4$ equals $h^0(\mathcal{O}_{\mathbb{F}_{1}}(e+2f))-1$, while the dimension of the space of curves of class $ae+bf$ on a given scroll equals $h^0(\mathcal{O}_{\mathbb{F}_{1}}(ae+bf))-1$.

{\fl Since} $\mathbb{F}_{1}$
is rational, we have $\chi(O_{\mathbb{F}_{1}})=1$. Therefore, Riemann-Roch implies that
\begin{equation}
\begin{split}
\chi(ae+bf)&=1+\frac{1}{2}((ae+bf)^{2}-(ae+bf)\cdot(-2e-3f)). \\
&=-\frac{1}{2}a^{2}+ab+ \frac{1}{2}a+ b+1.
\end{split} \notag
\end{equation}
Substituting $(a,b)=(1,2)$ yields
\begin{equation*}
\chi(e+2f)=5.
\end{equation*}

{\fl A} straightforward calculation (see \cite[Section $2$]{Cot} for details) now yields
\begin{equation*}
h^i(\mathcal{O}_{\mathbb{F}_{1}}(ae+bf))=0, i=1,2,
\end{equation*}
when $(a,b)=(1,2)$. On the other hand, every Castelnuovo curve is ACM, hence its ideal sheaf has vanishing higher cohomology. It follows immediately that curves of class $4e+7f$ on $\mathbb{F}_1$ are nonproblematic.

{\fl Finally,} we count curves $C$ lying on cones $S_{0,3}$ over twisted
cubics. To do so, note first that every $S_{0,3}$ is the image of the Hirzebruch
surface $F_3=\mathbb{P}(\mathcal{O}_{\mathbb{P}^1} \oplus \mathcal{O}_{\mathbb{P}^1}(3))$
under the map $\phi_{0,3}$ defined by the complete linear series $|\mc{O}_{\mb{F}_3}(e+3f)|$, where $e$ is the
divisor class of the  of $-3$-curve on $\mb{F}_3$ and $f$ is the class of the fiber. Note that $\phi_{0,3}$ is a
birational map that blows down the $-3$-curve to the vertex
of $S_{0,3}$.

{\fl The} intersection pairing on $\mb{F}_3$ is given by
\begin{equation*}
e^2=-3, e \cdot f=1, f^2=0.
\end{equation*}
The
canonical class of $\mb{F}_3$ is
\begin{equation*}
K_{F_3}=-2e-5f.
\end{equation*}

{\fl Now} say that $C$
avoids the vertex $S_{0,3}$. Then
$[\phi_{0,3}^{-1}C]=ae+bf$, for some integers $a$ and $b$. Since $\deg(C)=11$, we have
\begin{equation*}
(e+3f) \cdot (ae+bf)=11,
\end{equation*}
whence,
\begin{equation*}
b=11.
\end{equation*}

{\fl On} the other hand, the adjunction formula
then implies that
\begin{equation}\label{S_{0,3}eq2}
2g-2=((a-2)e+(b-5)f)\cdot (ae+bf).
\end{equation}
Substituting $g=12$ and $b=11$ in \eqref{S_{0,3}eq2} and solving for $a$ yields $(a,b)=(4,11)$. But
\[
(4e+11f) \cdot e=-1,
\]
which is only possible if $C$ and the $(-3)$-curve on $\mb{F}_3$ have components in common, which is absurd.

{\fl Finally,} say that $C$ passes through the vertex of $S_{0,3}$. The proper transform of $C$ on $\mb{F}_3$ has class
\[
\widetilde{a}e+11f
\]
for some nonnegative integer $\widetilde{a}$. There is a natural morphism $\phi$ from the space of stable maps $\mc{M}_0(\mb{F}_3)$ to the space of stable maps $\mc{M}_0(S_{0,3})$ induced by contraction of the exceptional curve on $\mb{F}_3$. In particular, every map $\psi: \mb{P}^1 \ra S_{0,3}$ whose image passes through the vertex of $S_{0,3}$ lifts via $\phi$ to a map $\widetilde{\psi}: \mb{P}^1 \ra \mb{F}_3$ whose image is the proper transform of $\psi(\mb{P}^1)$.

Here
\[
\beg{split}
\deg(\mc{N}_{\widetilde{\psi}}/(\mc{N}_{\wt{\psi}})_{\mbox{tors}})) &\leq \deg(-\wt{\psi}^* K_{\mb{F}_3}) -2 \\
& \leq (2e+5f) \cdot (\widetilde{a}e+11f) - 2 \\
&= 22 -\widetilde{a}-2.
\end{split}
\]

We conclude immediately that $\psi$ is nonproblematic.

\section{Curves spanning hyperplanes}\label{sec-three}
We now turn our attention to rational curves $C$ whose linear spans
are $3$-dimensional hyperplanes $H \subset \mb{P}^4$. By \eqref{1.2b} of Theorem~\ref{1.1},
in order to prove Theorem~\ref{1.1}, we must show the following
result.

\begin{thm}\label{2.3}
The reduced irreducible rational curves verifying
\[
g+i \geq 8+\min(g,5)
\]
determine a sublocus of $M_{11}$ of codimension greater than $g+i$.
\end{thm}
Much as in the nondegenerate case, we will call $C$ {\bf nonproblematic} if either
\beg{itemize}
\item[(i)] $C$ belongs to a sublocus of $M_{11}$ of codimension greater than $g+i$, or
\item[(ii)] $g+i < 8+\min(g,5)$.
\end{itemize}

\beg{proof}
We argue much as we did in our analysis of
nondegenerate curves. The new analysis is simpler, so we sketch it.

Fix $C$, and let $H$ denote the linear span of $C$. As in
\cite[Section $2.5$]{Cot}, choosing a $2$-plane $H_1 \sub H$ that is general
with respect to $C$ and choosing coordinates $x_0,\dots,x_4$ for $\mb{P}^4$
in such a way that $H$ and $H_1$ are defined by $x_4=0$ and $x_3=0$, respectively, we obtain $\mbox{gin}(\mc{I}_{C
  \cap H_1/H_1})$ of $C$, the hyperplane gin of $C$. For ease of notation, we'll denote the latter by
$\mbox{gin}(\mc{I}_{\Ga})$. The hyperplane gin is a monomial ideal in
$x_0$ and $x_1$, so it has a minimal generating set of the form
\[
(x_0^{k}, x_0^{k-1}x_1^{\lambda_{k-1}}, \dots,
  x_0x_1^{\lambda_{1}},x_1^{\lambda_{0}}). \notag
\]

By a result of Gruson and Peskine ~\cite[Cor. 4.8,
p.160]{Gr}, the invariants $\lambda_i$ of the above generating set
verify
\beg{equation}\label{gp}
\lambda_i - 1 \geq \lambda_{i+1} \geq \lambda_i - 2
\end{equation}
for all $i = 0, \dots, k-2$.

Moreover, $\mbox{gin}(\mc{I}_{\Ga})$ is Borel-fixed, so it has a
unique tree-representation analogous to the tree-representations for
the hyperplane gins of nondegenerate curves in $\mb{P}^4$ introduced
in Section $1$. Each tree may be obtained from an empty tree with a
single vertex by applying a sequence of rules that we denote by
$\Lambda$-rules; see \cite[Table $3$]{Cot}. Likewise,
$\mbox{gin}(\mc{I}_{C/H})$ is a Borel-fixed ideal whose tree of minimal
generators is obtainable from the tree associated to
$\mbox{gin}(\mc{I}_{\Ga})$ by finitely many $C$-rules, as given in \cite[Table $4$]{Cot}.

A straightforward inductive argument on trees shows that $\sum \lambda_i$
equals $11$, the degree of $C$. Moreover, by Ballico's result
\cite{Ba}, $\mc{I}_{\Ga}$ is $6$-regular. These two constraints, along with \eqref{gp}, force
\[
\mbox{gin}(\mc{I}_{\Ga})=\mbox{Borel }(x_1^5,
x_1^3x_0) \text{ or } \mbox{gin}(\mc{I}_{\Ga})= \mbox{Borel }(x_1^5,x_0^2x_1^2,x_0^3).
\]
Let $g_{\Gamma}$ denote the genus of the cone in $H$ with vertex $(0,0,0,1)$ over
the zero-dimensional scheme defined by $\mc{I}_{\Ga}$ in $H_1$; a
calculation yields $g_{\Gamma}=14$ or $g_{\Gamma}=15$. It follows that
\begin{equation}\label{g+ibound2}
g+h^1(\mc{I}_{C/H}) \leq 14 \text{ or } g+h^1(\mc{I}_{C/H}) \leq 15,
\end{equation}
respectively.

On the other hand, the long exact sequence in cohomology associated to
\begin{equation*}
0 \rightarrow \mathcal{I}_{\mathbb{P}^4}(4) \rightarrow \mathcal{I}_{C/\mathbb{P}^4}(5)
\rightarrow \mathcal{I}_{C/H}(5) \rightarrow 0,
\end{equation*} 
shows that
\begin{equation*}
h^1(\mathcal{I}_{C/H}(5))=h^1(\mathcal{I}_{C/\mathbb{P}^4}(5)).
\end{equation*}

Applying a single $C$-rule to either of our possible hyperplane gins
yields a $6$-regular ideal, accordingly, our estimates
\eqref{g+ibound2} improve to
\[
g+i \leq 13 \text{ and } g+i \leq 14,
\]
respectively.

We now face a basic dichotomy: either 
\beg{enumerate}
\item $C$ is $9$-regular. Appealing to the RTB stratification of $M^3_{11}$, together with \cite[Prop. 1.2]{GLP}, we see that those elements of $M^3_{11}$ with 8-irregular images sweep out a sublocus of codimension at least 3, corresponding to RTB stratum $(5,4,2)$ and specializations thereof. In other words, such maps have  codimension at least 11 in $M_{11}$. On the other hand, inspection together with the connectedness result of \cite[Thm. $2.5$]{AT} yields $g+i \leq 12$ in every case. Moreover, our bound on $g+i$ will improve to $g+i \leq 10$, provided the following assertion holds: {\it those maps whose images lie along at least 3 linearly independent quartics are nonproblematic}. 

To prove the assertion, note that the images $C$ of such maps fall into two categories, namely:
\beg{enumerate}
\item Those that lie on irreducible reduced quadric or cubic surfaces.
\item Those that are contained in complete intersection curves of type $(4,4)$.
\end{enumerate}

In fact, curves that lie on quadrics are automatically precluded, because none of the curve gins under consideration admit any minimal generators of degree less than three. Similarly, if $C$ lies on a cubic surface, then its hyperplane gin is necessarily $\mbox{Borel}(x_1^5,x_0^2x_1^2,x_0^3)$, and $\mbox{gin}(\mc{I}_C)$ will contain a unique minimal cubic generator. Applying the connectedness result \cite[Thm. 2.5]{AT}, we deduce that $g+i \leq 9$, with equality possible only when $g=0$. In particular, $C$ is nonproblematic, except possibly when $g=0$.

Now recall from \cite[pp.628-629 and p.638]{GH} that there are three basic types of singular cubic surfaces: those with double lines, cones over nonsingular plane cubics, and cubics with isolated singularities that are not cones. Those with double lines arise as projections of smooth cubic scrolls in $\mb{P}^4$. Those that are cones are desingularized by $\mb{P}^1$-bundles over nonsingular plane cubics $E$.

Those cubic surfaces with isolated singularities that are not cones are desingularized by $\mb{P}^2$ blown up in six points in special position. Moreover, the desingularization 
\[
\widetilde{S}= \mbox{Bl}_{6 \text{ pts}} \mb{P}^2 \ra S
\]
is defined, in every case by the linear system $3l- \sum_{i=1}^6 E_i$ corresponding to (pullbacks of) plane cubics passing through the six points in question. Here $l$ is the class of a line pulled back from $\mb{P}^2$, while the classes $E_i$ are the exceptional divisors of the blow-up.

We already know that degree-11 rational curves that lie on smooth cubic scrolls in $\mb{P}^4$ are nonproblematic, so the same is true of their projected images.  On the other hand, if $C$ is a degree-11 rational curve on a cone over a nonsingular plane cubic, then the corresponding map $\mb{P}^1 \ra \mb{P}^4$ is necessarily a multiple cover of one of the rulings of the cone, which is absurd. Finally, to handle smooth rational curves $C$ that lie on images of $\mb{P}^2$ blown up in six points, we argue as follows.

Let 
\[
[C]= al+ \sum_{i=1}^6 b_i E_i;
\]
by adjunction, we have
\beg{equation}\label{res1}
-2= (K_{\widetilde{S}}+[C]) \cdot [C]= a(a-3) - \sum_{i=1}^6 (b_i+1)b_i.
\end{equation}
Moreover, since $C$ is a degree-11 curve on $S$, we have
\beg{equation}\label{res2}
(3l- \sum_{i=1}^6 E_i) \cdot (al+ \sum_{i=1}^6 b_i E_i)= 3a+ \sum_{i=1}^6 b_i=11.
\end{equation}
Further, the fact that
$[C] \cdot E_i \geq 0$ for all $i, 1 \leq i \leq 6$, implies that 
\beg{equation}\label{res3}
a \geq 4, \text{ and } b_i \leq 0
\end{equation}
for all $i, 1 \leq i \leq 6$.
Finally, the fact that $C$ is smooth implies that 
\beg{equation}\label{res4}
b_i \geq -1
\end{equation}
for all $i$. The four equations \eqref{res1}-\eqref{res4} have no common integer solutions, so we conclude that $C$ is nonproblematic.

To handle curves of the second type, say that $X$ is a complete intersection curve of type $(4,4)$, such that $X= C \cup C^{\pr}$, for some quintic curve $C^{\pr}$. By the Hartshorne--Schenzel theorem \cite[Thm. 5.3.1]{M}, we deduce that
\[
h^1(\mc{I}_{C/H}(t))= h^1(\mc{I}_{C^{\pr}/H}(4-t))
\]
for all nonnegative integers $t$. It follows, just as in our analysis of nondegenerate rational curves of degree 11 lying on five linearly independent hypercubics, that if $C$ is 6-irregular, then its residual $C^{\pr}$ contains a nonresidual subscheme $Y$. By Bertini's theorem, $Y$ lies in the base locus of the linear system $\mc{L}$ of quartics containing $C$; since $\mc{L}$ is at least 3-dimensional by assumption, it follows that $Y$ is properly contained in $C^{\pr}$. Whence, we have $m:=\deg(Y) \leq 4$. Now write $C^{\pr}= Y \cup Y^{\pr}$. On one hand, liaison yields
\beg{equation}\label{eq1}
\deg(C \cap C^{\pr})= 46-2g.
\end{equation}
On the other hand, Bezout's theorem implies that
\beg{equation}\label{eq2}
\deg(C \cap Y^{\pr}) \leq 20-4m;
\end{equation}
otherwise, every quartic containing $C$ contains a component of $Y^{\pr}$. Applying \eqref{eq1} and \eqref{eq2}, we deduce
\[
\deg(C \cap Y) \geq 26+ 4m-2g.
\]
Via an explicit dimension count analogous to the one we obtained in our analysis of nondegenerate rational curves of degree 11 lying on five linearly independent hypercubics, we deduce that the corresponding rational curves $C$ are nonproblematic.

\item $C$ is $9$-irregular. Then \cite[Thm.~3.1]{GLP} implies that $C$
  admits a $10$-secant line. Because rational curves $C$ of degree $11$ that span hyperplanes and admit
$10$-secant lines comprise a sublocus of $M^3_{11}$ of codimension at least $6$, by
Lemma~\ref{1a}, the sublocus of such curves has codimension
  at least 14 in $M_{11}$. On the other hand, using the connectedness theorem \cite[Thm. $2.5$]{AT}, we see immediately that $g+i<14$ and, whence, that $C$ is nonproblematic. 
\end{enumerate}
\end{proof}

\section{Reducible curves}\label{sec-four}
In this section we prove the following theorem, which extends
\cite[Thm 4.1]{Klei}.

\beg{thm}\label{reduciblecurves}
On a general quintic threefold in $\mathbb{P}^{4}$, there is no connected, reducible
and reduced curve of degree at most $11$ whose components are rational.
\end{thm}

\beg{proof}
Suppose, on the contrary, that such a curve $C$ exists. By the results
of \cite{Klei} and \cite{Cot} we may
assume $C$ has two components and is of degree $11$. Consider
one of them. By the result of \cite[Theorem 3.1]{Klei}, it is either smooth or a 
six-nodal plane quintic. If it is smooth, then, by
\cite[Cor. 2.5(3)]{Klei}, either it has degree 4 or more and spans $\mathbb{P}^{4}$ or it is a rational normal curve of
degree less than 4. We will prove that there can be
no such $C$.

To this end, we follow the proof of \cite[Thm 4.]{Klei}. Let $M_{a}^{'}$ denote the open subscheme of
the Hilbert scheme of $\mathbb{P}^{4}$ parameterizing the smooth irreducible curves of
degree $a$ that are rational normal curves if $a \leq 4$ and that span
$\mathbb{P}^{4}$ if $a \geq 4$. Denote the scheme parameterizing
six-nodal plane quintics in $\mathbb{P}^4$ by $N_{5}$. Let
$R_{a,b,n}$ and $S_{6,n}$ denote the subsets of
$M_{a}^{'} \times M_{b}^{'}$ (resp., $N_5 \times M_{6}^{'}$) of pairs $(A,B)$
such that $A \cap B$ has length $n$. Finally, let $I_{a,b,n}$ (resp., $J_n$) denote
the subset of $R_{a,b,n} \times \mathbb{P}^{125}$ (resp., $S_{6,n}
\times \mathbb{P}^{125}$) of triples
$(A,B,F)$ such that $A \cup B \subset F$. The $F$ that contain a
plane form a proper closed subset of $\mathbb{P}^{125}$; form its complement, and form
the preimages of this complement in $I_{a,b,n}$ (resp. $J_n$). Replace
$R_{a,b,n}$ and $S_{a,n}$ by the images of those preimages, and
replace $I_{a,b,n}$ and $J_n$ by the preimages of the new $R_{a,b,n}$
and $S_{6,n}$. Then given any pair
$(A,B)$ in $R_{a,b,n}$ or $S_{6,n}$, there is an $F$ that contains both $A$ and $B$, but
not any plane. It remains to
show that $I_{a,b,n}$ (resp. $J_n$)
has dimension at most 124 for $a+b =11$ and $n \geq 1$.

We note that the fiber of $I_{a,b,n}$ (resp., $J_n$) over a pair $(A,B)$ is a
projective space of dimension $h^{0}(\mathcal{I}_{C}(5))-1$, where $C$ is the reducible
curve $C=A \cup B$ and $\mathcal{I}_{C}$ is the ideal sheaf of the
corresponding subscheme of $\mathbb{P}^4$. Hence we have

\begin{equation}
\begin{split}
&\dim I_{a,b,n} \leq \dim R_{a,b,n} + 125 -
\min_{C}\{h^0(O_C(5))
- h^1(\mathcal{I}_C(5))\} \text{ and}\\
&\dim J_n \leq \dim S_{6,n}+125 -
\min_{C}\{h^0(O_C(5))
- h^1(\mathcal{I}_C(5))\}.
\notag
\end{split}
\end{equation}

Obviously, we have
\begin{equation}
h^{0}(\mc{O}_C(5)) \geq \chi(\mc{O}_C(5)), \notag
\end{equation}
which implies
\begin{equation}
\begin{split}
&h^0(\mc{O}_C(5)) \geq 5(a+b)+2-n = 57-n \text{ and} \\
&h^0(\mc{O}_C(5)) \geq
20+ 5(6) + 1-n=51-n, \text{ respectively.} \notag
\end{split}
\end{equation}

Our theorem is then a consequence of the following two lemmas.

\beg{lem}\label{lemma1}
For $a+b=11$ and $n \geq 1$,
\begin{subequations}
\begin{align}
&\dim R_{a,b,n} \leq 56-n \text{ and} \label{2.1a}\\
&\dim S_{a,n}
\leq 50-n. \label{2.1b}
\end{align}
\end{subequations}
\end{lem}

\beg{lem}\label{lemma2}
For $a+b=11$ and $n\geq 1$, 
\begin{equation}
h^1(\mathcal{I}_C(5)) = 0. \notag
\end{equation}
\end{lem}

\beg{proof}[Proof of Lemma~\ref{lemma1}] 
Our argument is nearly identical to the proof of \cite[Lem. $3.0.2$]{Cot}, so
we merely sketch it.

Let $(A,B)$ denote an arbitrary
pair in $R_{a,b,n}$ or $S_n$. Assuming $a\leq b$, we have $\deg A \leq
5$. As in \cite[Cor. $2.5$]{Klei}, we may assume that the restricted
tangent bundles $T_{\mb{P}^4}|_A= \oplus_{i=1}^4
\mc{O}_{\mb{P}^1}(a_i)$ and $T_{\mb{P}^4}|_B=  \oplus_{i=1}^4
\mc{O}_{\mb{P}^1}(b_i)$ have
generic splitting types, i.e., that the sum of the absolute
values of differences $\sum_{i \neq j} |a_i-a_j|$ and $\sum_{i \neq j}
|b_i-b_j|$, respectively, are minimized. As in \cite[proof of
  Lem. $3.0.2$]{Cot}, we will repeatedly use the Key Fact cited in Section~\ref{ci333}.
  
\begin{enumerate}
\item $\deg A=1$. The required bound on $\dim R_{a,b,n}$ holds on
  the basis of the arguments of \cite[Lemma 4.2]{Klei}.
\item $\deg A=2$. On the basis of the argument in \cite[Lemma
  4.2]{Klei}, we may assume $7 \leq n \leq 8$. To bound $\dim
  R_{2,9,7}$ (resp., $\dim R_{2,9,8}$), it suffices to show that
  rational nondegenerate nonics admitting $7$-secant conics $B$
  (resp., $8$-secant conics) have codimension at least $8$ (resp., $9$) inside
  the generic splitting stratum for restricted tangent bundles. This
  follows easily, however, from
\begin{equation*}
T_{\mathbb{P}^4}|_B \cong \mc{O}_{\mathbb{P}^1}(12)
  \oplus \mc{O}_{\mathbb{P}^1}(11)^{\oplus 3},
\end{equation*}
together with the key fact cited above.
\item $\deg A=3$. On the basis of the argument in \cite[Lemma
  4.2]{Klei}, it's enough to verify the required bound for $\dim
  R_{3,8,8}$. For that purpose, it's in turn enough to check that
  rational nondegenerate octics $B$ admitting 8-secant twisted cubics have
  codimension at least 9 inside the generic splitting stratum. This
  follows from the fact that
\begin{equation*}
T_{\mathbb{P}^4}|_B \cong \mc{O}_{\mathbb{P}^1}(10)^{\oplus 4}.
\end{equation*}
\item $\deg A=4$. On the basis of the argument in \cite[Lemma
  4.2]{Klei}, we may assume $12 \leq n \leq 14$. Now observe that
  $\dim R_{4,7,n} \leq \dim R_{4,7,n}$ for all $n>7$; moreover, since
\begin{equation*}
T_{\mathbb{P}^4}|_B \cong \mc{O}_{\mathbb{P}^1}(9)^{\oplus 3}
  \oplus \mc{O}_{\mathbb{P}^1}(8),
\end{equation*}
those $B$ intersecting $A$ in $8$ points cut out a sublocus of
codimension $16$ inside $M_7$, and we conclude no unions $A \cup B \in
R_{4,7,n}$ lie on a general quintic.
\item $\deg A=5$. We have
\begin{equation*}
T_{\mathbb{P}^4}|B \cong \mc{O}_{\mathbb{P}^1}(8)^{\oplus 2}
  \oplus \mc{O}_{\mathbb{P}^1}(7)^{\oplus 2}.
\end{equation*}
First say that $A$ is smooth. Then $A$ and $B$ are $3$-regular, by
\cite[Cor. $2.5$]{Klei}, so they
are cut out by hyperquadrics and hypercubics and, moreover,
$h^1(\mc{I}_A(2))=0$. Whence,
\begin{equation*}
h^0(\mc{I}_A(2))=h^0(\mc{O}_{\mathbb{P}^4}(2))-h^0(\mc{O}_A(2))=15-(5(2)+1)=4,
\end{equation*}
by the Riemann--Roch theorem.
Observe that there is some hyperquadric containing $A$ but not $B$. To see
this, note first that by Bezout's theorem, no three linearly
independent hyperquadrics $Q_1, Q_2, Q_3$
containing $A$ may cut out a complete intersection curve containing $A
\cup B$. On the other hand, if $A \subset Q_1 \cap Q_2 \cap Q_3$ and
$Q_1, Q_2, \text{ and }Q_3$ do not cut out a complete intersection,
then $A$ lies on a cubic scroll cut out
by three hyperquadrics. The intersection of the scroll with any fourth
linearly independent hyperquadric is a sextic curve, which cannot
contain $A \cup B$. So we deduce once more that
there is some hyperquadric $Q$ containing $A$ but not $B$.

Thus $\deg A \cap B \leq \deg B \cap Q=12$. On the other hand, $T_{\mathbb{P}^4}|_A$ has balanced splitting
type $(7,6,6,6)$, from which it follows (by the same argument used
earlier) that for all $1 \leq n \leq
7$, the codimension of curves $A$ intersecting curves $B$ in
subschemes of length $n$ is equal to $2n$. As $\dim R_{5,6,m}
\leq \dim R_{5,6,n}$ whenever $m \geq n$, it
follows immediately that no curve in any $R_{5,6,n}, n\geq 1$ lies on
a general quintic hypersurface.

\item Finally, say $A$ is a six-nodal plane quintic and $B$ a smooth
sextic. Letting $J$ denote
the plane of $A$, it's clear that $A \cap B \subset J \cap B$, so $A
\cap B$ has degree at most 6. On the other hand, the restricted
tangent bundle $T_{\mathbb{P}^2}|_A$ has splitting type $(a_1,a_2)$,
where $a_1+a_2=15$. Assume $a_1 \geq a_2$. As usual, our goal is to
bound the codimension of those six-nodal plane quintics that meet
other six-nodal plane quintics in at most $5$ points by exploiting the
generic splitting of the restricted tangent bundle of $A$. In
particular, we are done provided $a_2 \geq 3$, which is certainly the
case. Indeed, pulling back the (dual of the) Euler sequence on
$\mb{P}^4$ to $A \cong \mb{P}^1$ and twisting by
$\mc{O}_{\mb{P}^1}(5)$ yields an exact sequence
\beg{equation*}
0 \ra \oplus_{i=1}^2 \mc{O}_{\mb{P}^1}(-a_i+5) \ra
V \otimes \mc{O}_{\mb{P}^1} \ra
\mc{O}_{\mb{P}^1}(5) \ra 0
\end{equation*}
where $V \sub H^0(\mc{O}_{\mb{P}^1}(5))$, which implies that $a_i-5 \geq 0$.
The proof of the first lemma is now complete.  
\end{enumerate}
\end{proof}

\beg{proof}[Proof of Lemma~\ref{lemma2}]
We now proceed to the proof of the second lemma, i.e., that all curves
$C$ in rational components $A$ and $B$ satisfying our hypotheses are
$6$-regular. By a result of G. Caviglia \cite[Thm. 2.1]{Gi}, we have
\begin{equation*}
\mbox{reg}(A \cup B) \leq \mbox{reg}(A)+ \mbox{reg}(B),
\end{equation*}
so in particular we are done whenever
\begin{equation*}
\mbox{reg}(A)+\mbox{reg}(B) \leq 6.
\end{equation*}
For $a \geq 2$, this follows immediately from the result of [Prop
  2.2]\cite{Klei2}. For $a=1$, we note that {\it a generic rational
  curve $B$ of degree $10$ in $\mb{P}^4$ is $4$-regular}, as one can
prove using Macaulay2 (see Section~\ref{sec-five}). But no preimage of a proper closed subset of $M_{10}$ dominates the space of quintics $F$, so we may assume without loss of generality that $B$ is $5$-regular. Because the line $A$ is $1$-regular, it follows once more that any reducible union $A \cup B \subset F$ is $5$-regular, and the lemma follows. 
\end{proof}
Theorem~\ref{reduciblecurves} follows immediately from Lemmas~\ref{lemma1} and \ref{lemma2}.
\end{proof}

\section{Auxiliary results}\label{sec-five}

In the preceding sections, we make use of the following result, which extends \cite[Lem. 3.4a]{Klei} when $d=11$. Complete proofs may be found in the earlier preprint version of this paper \cite{Cot2}.

\beg{thm}\label{codlemma}
The codimension of $M^4_{11,g}$ in $M_{11}$ is at least $\min(2g,12)$.
\end{thm}

\beg{proof} (Sketch.)
For any $f \in M^4_{11,g}$, $g$ is equal to the sum of the delta-invariants of the singularities of the image of $f$. Since an ordinary node imposes 2 conditions and has delta-invariant 1, we {\it expect} $M^4_{11,g}$ to have codimension at least $2g$ in $M^4_{11,g}$ when $g$ is small. This expectation turns out to be correct, and follows from two technical lemmas, each of which may be verified via a case-by-case analysis in local coordinates. 
\end{proof}

\beg{lem}Assume that $f \in M^4_{11,\de}$ has a unibranch singularity with $\de$-invariant $\de \le 6$. Deformations that preserve the ramification type of the singularity, together with its location and the location of its preimage determine a sublocus of $M_{11}$ of the expected codimension.
\end{lem}

\beg{lem}Assume that $f \in M^4_{11,\de}$ has a singularity, all of whose branches are smooth, with $\de$-invariant $\de \le 6$. Deformations that preserve the singularity, together with its location and the locations of its preimages determine a sublocus of $M_{11}$ of codimension at least $\de$.
\end{lem}

Finally, we prove the following results, which were invoked in the course of the proof of Theorem ~\ref{1.1bis}.

\beg{claim}\label{aux1}
A generic rational curve of degree 10 in $\mb{P}^4$ is 4-regular.
\end{claim}
\beg{claim}\label{aux2}
A generic rational curve of degree 11 in $\mb{P}^4$ of RTB splitting type $(4,3,2,2)$ is 4-regular.
\end{claim}
\beg{claim}\label{aux3}
A generic rational curve of degree 11 in $\mb{P}^4$ of RTB splitting type $(4,4,2,1)$ is 6-regular.
\end{claim}

\beg{proof}[Proof of Claim~\ref{aux1}.]
By upper-semicontinuity of cohomology, it suffices to produce a single degree-10 morphism $f: \mb{P}^1 \ra \mb{P}^4$ with the desired property. We accomplish this with the following Macaulay2 \cite{GS} program.
{\fl \url $kk=ZZ/32003$ \\
$\mbox{ringP1}=kk[t,u]$ \\
$\mbox{ringP4}=kk[x0,x1,x2,x3,x4]$ \\
$g0=t^{10}+t^9*u+t*u^9+u^{10}$ \\
$g1=7*t^{10}+101*t^8*u^2+ 355*t^5*u^5+ 999*u^{10}$ \\
$g2=29*t^{10}+ 99*t^3*u^7+ 67*t^2*u^8+ 83*u^{10}$ \\
$g3=61*t^{10}+79*t^5*u^5+t^3*u^7+901*t*u^9+ 53*u^{10}$ \\
$g4=741*t^{10}+t^8*u^2+ t^7*u^3+ t^6*u^4+ t^4*u^6+ t^2*u^8+ 9001*u^{10}$ \\
$g=\mbox{map}(\mbox{ringP1},\mbox{ringP4},{g0,g1,g2,g3,g4})$ \\
$J=\ker g$ \\
$\mbox{leadTerm}(J)$
}

The resulting monomial ideal (i.e., the revlex initial ideal of the image of the map $\mb{P}^1 \st{(g_0,\dots,g_4)}{\longrightarrow} \mb{P}^4$) is minimally generated in degrees at most 4, which implies the first assertion.
\end{proof}

Similarly, to check the second and third assertions, first note that given any degree-$d$ morphism $f: \mb{P}^1 \ra \mb{P}^4$, the $(-1)$-twist of the dual of the Euler sequence on $\mb{P}^4$ pulls back via $f$ to
\[
0 \ra f^*(\Om_{\mb{P}^4}(1)) \cong \oplus_{i=1}^4 \mc{O}(-a_i) \ra V \otimes \mc{O} \ra \mc{O}(d) \ra 0.
\]
where $V \sub H^0(\mc{O}_{\mb{P}^1}(d))$ is the five-dimensional space of sections defining $f$. In other words, the dual of $f^*(T_{\mb{P}^4}(-1))$ is the syzygy sheaf of $V$. Accordingly, to check the second and third assertions, we choose {\it syzygies} associated with splitting types $(4,3,2,2)$ and $(4,4,2,1)$ at random, and check that the corresponding vector spaces $V$ define morphisms to $\mb{P}^4$ whose images have the desired properties.

\beg{proof}[Proof of Claim~\ref{aux2}.]
By upper-semicontinuity, it suffices to construct a single degree-11 rational curve that is 4-regular and of the desired splitting type. To this end, begin by choosing sections $f_0, \dots, f_4 \in H^0(\mc{O}_{\mb{P}^1}(11))$ such that
\beg{equation}\label{syzeqns}
\beg{split}
&t^4 f_0 + t^3u f_1+ t^2u^2 f_2+ tu^3 f_3+ u^4 f_5=0, \\
&u^3f_0+ u^2t f_1+ ut^2 f_3+ t^3 f_4=0, \\
&ut f_1+ (u^2-t^2) f_2+ t^2 f_3+ u^2 f_4=0, \text{ and }\\
&t^2 f_0+ (t^2+tu+u^2) f_1+ tu f_2+ u^2 f_3+ (t^2-u^2) f_4=0.
\end{split}
\end{equation}

Now write
\[
f_i= \sum_{j=0}^{11} a_{i,j} t^i u^{11-j}
\]
where $t$ and $u$ are homogeneous coordinates on $\mb{P}^1$. The equations \eqref{syzeqns} impose 59 linearly independent conditions on the coefficients $a_{i,j}$. In fact, up a constant multiple, they force
\[
\beg{split}
&f_0= -t^9u^2- t^8u^3+ t^6u^5+ 2t^5u^6- 2t^3u^8+ tu^{10}, \\
&f_1=-t^8u^3+ 2t^7u^4- 3t^6u^5- t^5u^6+ 2t^4u^7+ 3t^3u^8+ t^2u^9- tu^{10}- u^{11}, \\
& f2= t^{11}+t^9u^2- 2t^8u^3+ 2t^7u^4- t^6u^5- 3t^4u^7- t^3u^8+ t^2u^9+ 2tu^{10}, \\
&f_3= t^{11}- t^8u^3+ 2t^7u^4- t^6u^5- t^4u^7- 3t^3u^8- 2t^2u^9+ tu^{10}+ u^{11}, \text{ and }\\
&f_4= -t^{10}u+ t^7u^4+ 3t^4u^7+ t^3u^8- t^2u^9- tu^{10}.
\end{split}
\]

Using Macaulay2, we find that the revlex initial ideal of the image of the map $f:\mb{P}^1 \st{(f_0,\dots,f_4)}{\longrightarrow} \mb{P}^4$ is $\mbox{Borel}(x_2^4,x_1x_2^2x_3,x_0^3)$, which is generated in degrees at most 4. We also find that the betti diagram of the minimal resolution of the ideal $I=(f_0, \dots, f_4)$ is given by
\[
\beg{split}
\mbox{total}:  \hspace{5pt} &1 \hspace{5pt} 5 \hspace{5pt} 4 \\
          0: \hspace{5pt} &1  \hspace{5pt}.  \hspace{5pt} . \\
          1: \hspace{5pt} &.  \hspace{5pt}.  \hspace{5pt} . \\
         & \vdots \\
         10:  \hspace{5pt} &.  \hspace{5pt} 5  \hspace{5pt} . \\
         11:  \hspace{5pt} &.  \hspace{5pt}.  \hspace{5pt} 2 \\
         12:  \hspace{5pt} &.  \hspace{5pt}.  \hspace{5pt} 1 \\
         13:  \hspace{5pt} &.  \hspace{5pt}.  \hspace{5pt} 1
\end{split}
\]
Therefore, $f$ has RTB splitting type $(4,3,2,2)$, and Claim~\ref{aux2} follows.
\end{proof}

\beg{proof}[Proof of Claim~\ref{aux3}]
Much as in the proof of Claim~\ref{aux2}, we begin by choosing sections $f_0, \dots, f_4 \in H^0(\mc{O}_{\mb{P}^1}(11))$ that verify
\beg{equation}\label{syzeqns2}
\beg{split}
&t^4f_0+t^3uf_1+t^2u^2 f_2+tu^3f_3+u^4f_4=0, \\
&u^4f_0+u^3tf_1+(t^4+u^2t^2)f_2+ut^3f_3+t^4f_4=0, \\
&utf_1+(u^2-t^2)f_2+t^2f_3+u^2f_4=0, \text{ and }\\
&tf_0+(t+u)f_2+(t-u)f_3+uf_4=0.
\end{split}
\end{equation}

Up to a constant multiple, the latter relations force
\[
\beg{split}
&f_0=2t^{10}u-t^9u^2-4t^7u^4-t^6u^5+2t^5u^6-t^4u^7+3t^3u^8+t^2u^9-tu^{10}, \\
&f_1=-2t^{11}+t^{10}u+2t^9u^2+3t^8u^3-2t^7u^4-t^6u^5-t^5u^6+t^4u^7+2t^3u^8-2t^2u^9-tu^{10}+u^{11}, \\
&f_2=-2t^{10}u+t^9u^2+t^8u^3+t^7u^4+t^6u^5-t^5u^6-tu^{10}, \\
&f_3=-t^8u^3-2t^7u^4-t^6u^5+t^5u^6+4t^4u^7+t^3u^8-2t^2u^9-tu^{10}, \text{ and }\\
&f_4=2t^10u-t^9u^2+3t^8u^3-2t^7u^4-4t^6u^5-t^5u^6+2t^3u^8+t^2u^9.
\end{split}
\]

Using Macaulay2, we find that the revlex initial ideal of the image of the map $f:\mb{P}^1 \st{(f_0,\dots,f_4)}{\longrightarrow} \mb{P}^4$ is 
\[
\beg{split}
I=&(x_0^2x_2,x_0^3,x_0x_1^2x_2,x_0^2x_3^3,x_2^3x_3,x_1x_2^2x_3,x_0x_2^2x_3,x_1^2x_2x_3,x_0x_1x_2x_3,x_0x_1^2x_3, \\
&x_0^2x_1x_3,x_2^4,x_1x_2^3,x_0x_2^3,x_1^2x_2^2,x_0x_1x_2^2,x_1^3x_2,x_0x_1^3,x_0^2x_1^2,x_0x_1x_3^3,x_1^3x_3^3),
\end{split}
\]
which is generated in degrees at most 6. Note, curiously, that $I$ is {\it not} Borel-fixed, so it is not the revlex gin of the image of $f$. Moreover, a calculation yields that the genus of $\mbox{Im}(f)$ is equal to 1, so the curve is singular! These facts are, however, of no concern to us.

Finally, the betti diagram of the minimal resolution of the ideal $I=(f_0, \dots, f_4)$ is given by
\[
\beg{split}
\mbox{total}:  \hspace{5pt} &1 \hspace{5pt} 5 \hspace{5pt} 4 \\
          0: \hspace{5pt} &1  \hspace{5pt}.  \hspace{5pt} . \\
          1: \hspace{5pt} &.  \hspace{5pt}.  \hspace{5pt} . \\
         & \vdots \\
         10:  \hspace{5pt} &.  \hspace{5pt} 5  \hspace{5pt} 1 \\
         11:  \hspace{5pt} &.  \hspace{5pt}.  \hspace{5pt} 1 \\
         12:  \hspace{5pt} &.  \hspace{5pt}.  \hspace{5pt} . \\
         13:  \hspace{5pt} &.  \hspace{5pt}.  \hspace{5pt} 2
\end{split}
\]
Whence, $f$ has RTB splitting type $(4,4,2,1)$, and Claim~\ref{aux3} follows.
\end{proof}

{\fl \bf \small Laboratoire de Math\'ematiques Jean Leray \\
Unit\'e mixte de recherche 6629 du CNRS\\
Universit\'e de Nantes\\
2 rue de la Houssini\`ere, BP 92208 \\
44322 Nantes Cedex 3, FRANCE \\

{\it Email address}: \url{cotteril@math.harvard.edu} \\
{\it Web page}: \url{www.mast.queensu.ca/~cotteril}}

\end{document}